\documentclass[smallextended]{svjour3}
\usepackage{amsmath}
\usepackage{amssymb}
\usepackage{amsfonts}
\usepackage{algorithm}
\usepackage{graphicx}
\usepackage{color}
\usepackage{epsfig}
\usepackage{mathptmx}
\usepackage{latexsym}

\RequirePackage{fix-cm}
\smartqed

\newcommand{\bff}{\mathbf f}

\newcommand{\bx}{\mathbf x}
\newcommand{\by}{\mathbf y}

\newcommand{\zero}{\mathbf 0}

\begin{document}

\title{A positivity preserving inexact Noda iteration for computing the
smallest eigenpair of a large irreducible $M$-matrix\thanks{%
The first author was supported in part by the National Basic Research
Program of China 2011CB302400 and the National Science Foundation of China
(No. 11371219), and the second and third authors were supported in part by
the National Science Council, the National Center for Theoretical Sciences,
the Center of MMSC, and ST Yau Center at Chiao-Da in Taiwan.}}
\author{Zhongxiao Jia \and Wen-Wei Lin \and Ching-Sung Liu }
\thanks{The first author was supported in part by the National Basic
Research Program of China 2011CB302400 and the National Science Foundation
of China (No. 11371219), and the second and third authors were supported in
part by the National Science Council, the National Center for Theoretical
Sciences, the Center of MMSC, and ST Yau Center at Chiao-Da in Taiwan.}
\date{Received: date / Accepted: date}
\maketitle

\begin{abstract}
In this paper, based on the Noda iteration, we present inexact Noda
iterations (INI), to find the smallest eigenvalue and the associated
positive eigenvector of a large irreducible nonsingular $M$-matrix. The
positivity of approximations is critical in applications, and if the
approximations lose the positivity then they may be meaningless and could
not be interpreted. We propose two different inner tolerance strategies for
solving the inner linear systems involved, and prove that the convergence of
resulting INI algorithms is globally linear and superlinear with the
convergence order $\frac{1+\sqrt{5}}{2}$, respectively. The proposed INI
algorithms are structure preserving and maintains the positivity of
approximate eigenvectors. We also revisit the exact Noda iteration and
establish a new quadratic convergence result. All the above is first done
for the problem of computing the Perron root and the positive Perron vector
of an irreducible nonnegative matrix and is then adapted to computing the
smallest eigenpair of the irreducible nonsingular $M$-matrix. Numerical
examples illustrate that the proposed INI algorithms are practical, and they
always preserve the positivity of approximate eigenvectors. We compare them
with the Jacobi--Davidson method, the implicitly restarted Arnoldi method
and the explicitly restarted Krylov--Schur method, all of which cannot
guarantee the positivity of approximate eigenvectors, and illustrate that
the overall efficiency of the INI algorithms is competitive with and can be
considerably higher than the latter three methods.

\end{abstract}



\titlerunning{A positivity preserving inexact Noda iteration}


\institute{Zhongxiao Jia \at
              Department of Mathematical Sciences, Tsinghua University, Beijing 100084,
              People's Republic of China \\
              \email{jiazx@tsinghua.edu.cn}                      \and
           Wen-Wei Lin \at
              Department of Applied Mathematics, National Chiao Tung University, Hsinchu
              300, Taiwan \\
              \email{wwlin@math.nctu.edu.tw}
           \and
           Ching-Sung Liu \at
              Department of Mathematics, National Tsinghua University, Hsinchu 300, Taiwan \\
              \email{chingsungliu@gmail.com}
           }


\section{Introduction}

Irreducible nonsingular $M$-matrices are one class of the most important
matrices from applications, such as discretized PDEs, Markov chains \cite%
{Abate94,Rober90} and electric circuits \cite{Shi96}, and have been studied
extensively in the literature; see, for instance, \cite[Chapter 6]{BPl94}.
In many applications, one is interested in finding the smallest eigenvalue $%
\lambda$ and the associated eigenvector $\mathbf{x}$ of an irreducible
nonsingular $M$-matrix $A\in \mathbb{R}^{n\times n}$.

$M$-matrices are closely related to nonnegative matrices. For instance, an $M
$-matrix $A$ can be expressed in the form $A=\sigma I-B$ with a nonnegative
matrix $B\geq 0$ and some constant $\sigma >\rho (B)$, the spectral radius
of $B$; $A^{-1}$ is nonnegative. For more properties and a systematic
account of $M$-matrices and nonnegative matrices, see \cite{HJo85,BPl94}.

Nonnegative matrices have important applications in many areas \cite{BPl94},
including economics, statistics and network theory. For a nonnegative matrix
$B$, at least one of the eigenvalues of maximal magnitude is nonnegative and
hence equal to the spectral radius $\rho (B)$. The corresponding
eigenvectors $\mathbf{x}$ satisfy $B\mathbf{x}=\rho (B)\mathbf{x}$ and are
called the Perron vectors of $B$ if they are nonnegative. The nonnegative $B$
always has at least one Perron vector. In applications, the Perron vectors
play an important role, and they describe, e.g., an equilibrium, a
probability distribution or an optimal network property \cite{BPl94}.
Furthermore, if $B$ is irreducible, then the Perron-Frobenuis theorem~\cite%
{HJo85} states that $\rho(B)$ is simple and there is a positive eigenvector $%
\mathbf{x}$ associated with $\rho(B)$. One is often interested in verifying
the uniqueness and strict positivity of the Perron vector of the irreducible
nonnegative matrix $B$ \cite{HJo85,BPl94}. The well-known PageRank vectors
are special Perron vectors of very large Google matrices whose largest
eigenvalues are equal to one \cite{Lang06}. From the relation $A=\sigma I-B$%
, we see that the smallest eigenvalue $\lambda $ of the irreducible
nonsingular $M$-matrix $A$ is simple and equal to $\sigma -\rho (B)>0$, and
the positive vector $\mathbf{x}$ is the unique associated eigenvector up to
scaling. Consequently, if $\sigma$ is available, then the computation of the
smallest eigenpair $(\lambda,\mathbf{x})$ of $A$ amounts to that of the
largest eigenpair of $B$.

For a general large and sparse $A$, there are a number of general numerical
methods for computing a small number of its eigenpairs. Krylov type methods
including the power method applied to $A$ directly are suitable for exterior
eigenvalues, i.e., some eigenvalues close to the exterior of the spectrum,
and the associated eigenvectors, but they face serious challenges when the
desired eigenvalues are in the interior of the spectrum. Many methods, such
as inverse iteration, Rayleigh quotient iteration (RQI) and shift-invert
Arnoldi, have been developed to overcome these difficulties; see \cite%
{Saa92,Par98,Ste01}. However, they require the solution of a possibly
ill-conditioned large linear system involving a shifted $A$, called inner
linear system, at each iteration. This is generally very difficult and even
impractical by a direct solver since a factorization of a shifted $A$ may be
expensive or prohibited. When inner linear systems are approximately solved
by iterative solvers, we are led to inner-outer iteration methods, also
called inexact eigensolvers. The inner iteration means that the inner linear
system at each step is approximately solved iteratively, while the outer
iteration is the update of the approximate eigenpair(s). Throughout this
paper, we always make the underlying hypothesis that a direct solver is not
viable for large sparse linear systems and only iterative solvers can be
used.

There has been ever growing and intensive interest in inexact eigensolvers
over last two decades. Among them, inexact inverse iteration\cite%
{LLL97,BGS06,Liu12} and inexact RQI \cite{SimEld02,XueSzl11,Jia12,Jia11} are
the simplest and the most basic ones. In addition, they are often key
ingredients of other sophisticated and practical inexact methods, such as
inverse subspace iteration \cite{RSS09}, the Jacobi--Davidson method \cite%
{SVo96,Ste01} and the shift-invert residual Arnoldi method \cite{Lee07,Jia13}%
. In the mentioned papers and the references therein as well as some others,
a number of theoretical results have been established on these methods.
Particularly, it has been shown \cite{SimEld02,Jia12,Jia11} that, different
from RQI, the inexact RQI generally does not converges globally, rather it
is only locally convergent by requiring that the initial vector is a
reasonably good approximate eigenvector. Because of these two reasons, the
inexact RQI is seldom used in practice unless a good initial guess to the
desired eigenvector is already available.

For the computation of the Perron vector of the nonnegative $B$ and the
eigenvector of the $M$-matrix $A$ associated with the smallest eigenvalue, a
central concern is how to preserve the strict positivity of approximate
eigenvectors. For nonnegative matrices and $M$-matrices, such positivity is
crucial in some applications since if all the components of an approximate
eigenvector do not have the same sign then it may be physically meaningless
and could not be interpreted. Unfortunately, all the methods but the power
method mentioned previously are not structure preserving and cannot
guarantee the desirable positivity of approximations since it is possible
that a converged approximation of $\mathbf{x}$ may well have negative
components, as is the typical case when the unit length $\mathbf{x}$ has
very small components. Theoretically, the power method fits into this
purpose for nonnegative matrices and naturally preserves the strict
positivity of approximate eigenvectors, provided that the starting vector is
positive. Due to the equivalence of nonnegative matrix and $M$-matrix
eigenproblems, the power method can be adapted to computing the smallest
eigenpair of an irreducible nonsingular $M$-matrix. However, the power
method, though globally convergent, is generally very slow and may be
impractical. Therefore, it is very appealing in both theory and practice to
develop both efficient and reliable positivity preserving numerical methods
for the nonnegative matrix and $M$-matrix eigenproblems. We will devote
ourselves to this topic in this paper.

In $1971$, Noda \cite{Nod71} introduced an inverse iteration method with
variable shifts for nonnegative matrix eigenvalue problems. This iteration
method is a structure preserving method and was motivated by the works of
Collatz \cite{Collatz42} in 1942 and Wielandt \cite{Wiela50} in 1950; see
\cite[p. 37, 39 and 59]{varga}, \cite[p. 373]{GLo96} and \cite[p. 55]{BPl94}
for a description and historic overview. However, it was Noda who proposed
its inverse iteration form with the variable shifts different from Rayleigh
quotients. We, therefore, call the iteration the Noda iteration (NI). Given
a positive starting vector, NI naturally preserves the strict positivity of
approximate eigenvectors at all iterations. It has been adapted to computing
the smallest eigenpair of an irreducible nonsingular $M$-matrix \cite%
{Xue96,AXY02}. The purpose of \cite{Xue96,AXY02} is to compute the smallest
eigenvalue of such $M$-matrix with high relative accuracy. There, it has
been shown that NI is practical and effective for such a pursue, in which
the linear systems involved are solved accurately by a special direct
solver, called the GTH like algorithm.

It is well known that RQI is almost always globally convergent for any
starting vector \cite[Theorem 4.9.1]{Par98}, but its correct convergence to
a desired eigenpair is conditional and requires that the initial vector be a
reasonably accurate approximation to the desired eigenvector \cite{Par98}.
As a result, RQI itself is seldom used to solve a practical problem unless a
good initial guess is available. In contrast, a major advantage of NI is
that, for any positive initial vector, it converges globally and computes
the desired eigenpair correctly. Furthermore, the convergence order of NI is
asymptotically superlinear \cite{Nod71} and actually quadratic \cite{Els76}.
As it will become clear, NI always generates a monotonically decreasing
sequence of approximate eigenvalues whose convergence to $\rho(B)$ is
guaranteed; for the smallest eigenvalue $\lambda$ of the irreducible
nonsingular $M$-matrix $A$, it always generates a monotonically increasing
sequence of approximate eigenvalues that converge to $\lambda$
unconditionally. In other words, the approximate eigenvalues converge to $%
\rho(B)$ from above or converge to the smallest eigenvalue of the $M$-matrix
from below. In contrast, for symmetric matrices, Rayleigh quotients always
reside in the spectrum interval. As a result, if RQI converges correctly,
the sequence of approximations approach the Perron root of the nonnegative
matrix and the smallest eigenvalue of the $M$-matrix from inside to outside.

In this paper, keep in mind that only iterative solvers are supposed to be
viable to solve inner linear systems approximately. Based on NI, we first
propose an inexact Noda iteration (INI) to find the Perron root and vector
of an irreducible nonnegative matrix $B$. As an inexact eigensolver, our
major contribution is to propose two practical inner tolerance strategies
for solving the inner linear systems involved, so that the resulting two INI
algorithms are structure preserving and globally converge. The first inner
tolerance strategy uses $\gamma\min(\mathbf{x}_{k})$ as a stopping criterion
for inner iterations with the constant $\gamma<1$ and $\mathbf{x}_k$ is the
current positive approximate eigenvector. The second inner tolerance
strategy solves the inner linear systems with certain decreasing tolerances
for inner iterations, which will be described in the context. We establish a
rigorous convergence theory of INI with these two inner tolerance
strategies, proving that the convergence of the former iteration is globally
linear with the asymptotic convergence factor bounded by $\frac{2\gamma}{%
1+\gamma}$ and that of the latter is asymptotically superlinear with the
convergence order $\frac{1+\sqrt{5}}{2}$, respectively. In order to derive
this superlinear convergence order, we establish a close relationship
between the eigenvalue error and the eigenvector error obtained by NI and
INI, which is interesting in its own right. We also revisit the convergence
of NI and establish a new quadratic convergence result different from that
in \cite{Els76}. As we will see, the INI algorithms developed and the theory
established are easily extended to the computation of the smallest eigenpair
of an irreducible nonsingular $M$-matrix $A$.

Finally, we stress that, different from \cite{Xue96,AXY02}, our aim is the
positivity preserving computation of the desired eigenvectors, while their
concern is the relative high accuracy computation of the Perron root and the
smallest eigenvalue of an $M$-matrix.

The rest of this paper is organized as follows. In Section~2, we introduce
NI and some preliminaries. In Section~3, we present an inexact Noda
iteration and prove some basic properties of it. In Section~4, we propose
two practical INI algorithms, called INI\_1 and INI\_2, respectively, for
computing the spectral radius and the Perron vector of an irreducible
nonnegative matrix. We then establish their global convergence theory.
Moreover, we precisely derive the asymptotic linear convergence factor of
INI\_1 and superlinear convergence order of INI\_2. In Section~5, we adapt
INI\_1 and INI\_2 to the computation of the smallest eigenvalue and the
corresponding vector of an irreducible nonsingular $M$-matrix and present
their convergence theory. In Section~6, we report the numerical results on a
few practical problems to justify the convergence theory of INI and
illustrate their effectiveness. We also compare INI with the
Jacobi--Davidson method \cite{SVo96}, the implicitly restarted Arnoldi
method \cite{Sor92} and the explicitly restarted Krylov--Schur method \cite%
{Ste02}, all of which are not positivity preserving. The experiments
indicate that the proposed INI algorithms always preserve the positivity of
approximate eigenvectors, while the other three methods often fail to do so.
Also, we demonstrate that the INI algorithms are efficient, competitive with
and can outperform the other three methods considerably. Finally, we
summarize the paper with some concluding remarks in Section~7.

\section{Preliminaries, Notation and the Noda Iteration}

\label{prem}

\subsection{Preliminaries and Notation}

\label{sec:eigvec} For any real matrix $B=\left[ b_{ij}\right] \in \mathbb{R}%
^{n\times n}$, we denote $|B|=[|b_{ij}|]$. If the entries of $B$ are all
nonnegative (positive), then we write $B\geq 0$ $(>0)$. For real matrices $B$
and $C$ of the same size, if $B-C$ is a nonnegative matrix, we write $B\geq
C $. A nonnegative (positive) vector is similarly defined. A nonnegative
matrix $B$ is said to be reducible if there exists a permutation matrix $P$
such that%
\begin{equation*}
P^{T}\!BP=\left[
\begin{array}{cc}
E & F \\
O & G%
\end{array}%
\right],
\end{equation*}
where $E$ and $G$ are square matrices; otherwise it is irreducible. Here the
superscript $T$ denotes the transpose of a vector or matrix. Throughout the
paper, we use a $2$-norm for vectors and matrices, and all matrices are $%
n\times n$ unless specified otherwise.

We review some fundamental properties of nonnegative matrices and $M$%
-matrices.

\begin{lemma}[\cite{BPl94}]
\label{thm:M} Let $A\in
\mathbb{R}^{n\times n}$. Then the following statements are
equivalent:
\begin{enumerate}
\item $A= \left(a_{ij}\right)$, $a_{ij}\le 0$ for $i\neq j$, and $A^{-1}\ge
0 $;

\item $A=\sigma I-B$ with some $B\ge 0$ and $\sigma > \rho(B)$.
\end{enumerate}
Matrices having the above properties are called nonsingular $M$-matrices.
\end{lemma}

For a pair of positive vectors $\mathbf{v}$ and $\mathbf{w}$, define
\begin{equation*}
\max \left(\frac{\mathbf{w}}{\mathbf{v}}\right) = \underset{i}{\max}\left(%
\frac{\mathbf{w}^{(i)}}{\mathbf{v}^{(i)}}\right),\text{ \ }\min\left(\frac{%
\mathbf{w}}{\mathbf{v}}\right) = \underset{i}{\min}\left(\frac{\mathbf{w}%
^{(i)}} {\mathbf{v}^{(i)}}\right),
\end{equation*}
where $\mathbf{v}=[\mathbf{v}^{(1)},\mathbf{v}^{(2)},\ldots,\mathbf{v}%
^{(n)}]^T$ and $\mathbf{w}=[\mathbf{w}^{(1)},\mathbf{w}^{(2)},\ldots,\mathbf{%
w}^{(n)}]^T$. The following lemma gives bounds for the spectral radius of a
nonnegative matrix $B$; see \cite{BPl94,GLo96,HJo85,varga}.

\begin{lemma}[{{\protect\cite[p.~493]{HJo85}}}]
\label{maxmin}Let $B\ $be an irreducible nonnegative matrix. If $\mathbf{v}>
\mathbf{0}$ is not an eigenvector of $B$, then
\begin{equation}
\min\left(\frac{B\mathbf{v}}{\mathbf{v}}\right) <\rho(B) <\max \left(\frac{B%
\mathbf{v}}{\mathbf{v}}\right).  \label{eq:maxmin}
\end{equation}
\end{lemma}

Suppose that $A$ is an irreducible nonsingular $M$-matrix and $(\lambda,%
\mathbf{x})$ is the smallest eigenpair of it. Then if $\mathbf{v}>0$ is not
an eigenvector of $A$, it is easily justified from (\ref{eq:maxmin}) and $%
A=\sigma I-B$ that
\begin{equation*}
\min\left(\frac{A\mathbf{v}}{\mathbf{v}}\right) <\lambda <\max \left(\frac{A%
\mathbf{v}}{\mathbf{v}}\right).  \label{eq:minmax}
\end{equation*}

For an irreducible nonnegative matrix $B$, recall that the largest
eigenvalue $\rho(B)$ of $B$ is simple. Let $\mathbf{x}$ be the unit length
positive eigenvector corresponding to $\rho(B)$. Then for any orthogonal
matrix $\left[
\begin{array}{cc}
\mathbf{x} & V%
\end{array}%
\right] $ it holds (cf. \cite{GLo96}) that
\begin{equation}
\left[
\begin{array}{c}
\mathbf{x}^T \\
V^T%
\end{array}%
\right] B\left[
\begin{array}{cc}
\mathbf{x} & V%
\end{array}%
\right] = \left[
\begin{array}{cc}
\rho(B) & \mathbf{c}^T \\
0 & L%
\end{array}%
\right]  \label{eqn: partition}
\end{equation}%
with $L=V^TBV$ whose eigenvalues constitute the other eigenvalues of $B$. If
$\mu$ is not an eigenvalue of $L$,
the sep function for $\mu$ and $L$ is defined as
\begin{equation}
\mathrm{sep}(\mu, L)= \Vert (\mu I-L)^{-1}\Vert^{-1},  \label{eq:sep}
\end{equation}
which is well defined as $\mu\rightarrow\rho(B)$ since $\rho(B)$ is simple.
Throughout the paper, we will denote by $\angle(\mathbf{w},\mathbf{z})$ the
acute angle of any two nonzero vectors $\mathbf{w}$ and $\mathbf{z}$.

\subsection{The Noda iteration}

In \cite{Nod71}, Noda presented an inverse iteration with variable shifts
for computing the Perron root and vector of an irreducible nonnegative
matrix $B$. Given an initial guess $\mathbf{x}_0 >\mathbf{0}$ with $\Vert%
\mathbf{x}_0\Vert =1$, the Noda iteration (NI) is an inverse iteration with
variable shifts, and each iteration consists of three steps
\begin{align}
(\overline{\lambda}_kI-B) \, \mathbf{y}_{k+1} &=\mathbf{x}_k,
\label{eq:step1} \\
\mathbf{x}_{k+1} &=\mathbf{y}_{k+1} \, / \, \Vert \mathbf{y}_{k+1}\Vert,
\label{eq:step2} \\
\overline{\lambda}_{k+1} &=\max \left(\frac{B\mathbf{x}_{k+1}}{\mathbf{x}%
_{k+1}}\right).  \notag  \label{eq:step3}
\end{align}%
The main step is to compute a new approximation $\mathbf{x}_{k+1}$ to $%
\mathbf{x}$ by solving the inner linear system (\ref{eq:step1}). Lemma~\ref%
{maxmin} shows that $\overline{\lambda}_k>$ $\rho(B)$ as long as $\mathbf{x}%
_k$ is not a scalar multiple of eigenvector $\mathbf{x}$. Furthermore, since
$\overline{\lambda}_kI-B$ is an irreducible nonsingular $M$-matrix, its
inverse is irreducible nonnegative. Therefore, we have $\mathbf{y}_{k+1}>%
\mathbf{0}$ and $\mathbf{x}_{k+1}>\mathbf{0}$, meaning that the above
iteration scheme preserves the strict positivity of approximate eigenvector
sequence $\{\mathbf{x}_k\}$. We also see that NI is different from RQI,
where, in the symmetric case, the Rayleigh quotient of $B$ with respect to
any vector lies in the spectrum interval of $B$ and thus no more than $%
\rho(B)$, and the approximate eigenvectors obtained by RQI do not preserve
the strict positivity.

After variable transformation, we get the relation
\begin{equation}
\overline{\lambda}_{k+1}= \overline{\lambda}_k-\min\left(\frac{\mathbf{x}_k}{%
\mathbf{y}_{k+1}}\right),  \label{lambdaupdate}
\end{equation}
so $\overline{\lambda}_k$ is monotonically decreasing. Based on (\ref%
{eq:step1}), (\ref{eq:step2}) and (\ref{lambdaupdate}), we can present NI as
Algorithm~\ref{alg:iva}.

\begin{algorithm}
\begin{enumerate}
  \item Given an initial guess $\bx_0 > \zero$ with $\Vert \bx_0\Vert =1$ and ${\sf tol}>0$,
  compute $\overline{\lambda}_0=
  \max \left(\frac{B\bx_0}{\bx_0}\right)$.
  \item {\bf for} $k =0,1,2,\dots$
  \item \quad Solve the linear system $(\overline{\lambda}_kI-B) \, \by_{k+1}= \bx_k$.
  \item \quad Normalize the vector $\bx_{k+1}= \by_{k+1} \, / \, \Vert \by_{k+1}\Vert $.
  \item \quad Compute $\overline{\lambda}_{k+1}= \overline{\lambda}_k-\min\left(\frac{\mathbf{x}_k}{%
\mathbf{y}_{k+1}}\right)$.
  \item {\bf until} convergence: $\Vert B\bx_{k+1}-\overline{\lambda}_{k+1}
  \bx_{k+1}\Vert <{\sf tol}$.
\end{enumerate}
\caption{Noda iteration (NI)}
\label{alg:iva}
\end{algorithm}

\section{The inexact Noda iteration and some basic properties}

\label{sec:INVITTHM}

Since it is supposed that only iterative solvers are viable to solve the
linear system (\ref{eq:step1}) approximately at step 3 of Algorithm~\ref%
{alg:iva}, in this section, we first propose an inexact Noda iteration (INI)
for the eigenvalue problem of an irreducible nonnegative matrix $B$, and
then we prove a number of basic properties of INI, which will be used to
establish the convergence theory of two practical INI algorithms to be
proposed in Section 4.

\subsection{The inexact Noda iteration}

In INI we compute an approximate solution $\mathbf{y}_{k+1} $ in step 3 of
Algorithm~\ref{alg:iva}, such that%
\begin{align}
(\overline{\lambda }_{k}I-B)\,\mathbf{y}_{k+1}& =\mathbf{x}_{k}+\mathbf{f}%
_{k},\text{{}}  \label{eq:inexactsys} \\
\mathbf{x}_{k+1}& =\mathbf{y}_{k+1} \, / \, \Vert \mathbf{y}_{k+1}\Vert ,
\label{eq:normal}
\end{align}%
where $\mathbf{f}_{k}$ is the residual vector, whose norm $\xi_{k}:=\Vert
\mathbf{f}_{k}\Vert $ is bounded by the inner tolerance at iteration $k$.

\begin{lemma}
\label{monotone}Let $B$ be an irreducible nonnegative matrix and $0\leq
\gamma <1$ be a fixed constant. For the unit length $\mathbf{x}_{k}>\mathbf{0}$,
if $\mathbf{x}_{k}\not=\mathbf{x}$ and $\mathbf{f}_{k}$
in {\rm (\ref{eq:inexactsys})} satisfies
\begin{equation*}
|(\overline{\lambda }_{k}I-B)\,\mathbf{y}_{k+1}-\mathbf{x}_{k}|=|\mathbf{f}%
_{k}|\leq \gamma \,\mathbf{x}_{k},  \label{inexact condi}
\end{equation*}%
then the new approximation $\mathbf{x}_{k+1} >\mathbf{0}$ in
{\rm (\ref{eq:normal})} and the sequence $\{\overline{\lambda }_{k}\}$
with $\overline{\lambda }_{k}
=\max \left( \frac{B\mathbf{x}%
_{k}}{\mathbf{x}_{k}}\right) $ is monotonically decreasing and bounded below
by $\rho (B)$, i.e.,
\begin{equation}
\overline{\lambda }_{k}>\overline{\lambda }_{k+1}\geq \rho (B).
\label{eq:monolam}
\end{equation}
\end{lemma}

\begin{proof}
By $0\leq \gamma<1$  and $|\mathbf{f}%
_{k}|\leq \gamma \,\mathbf{x}_{k}$, it is known that $\mathbf{x}_{k}+%
\mathbf{f}_{k}>\mathbf{0}$. By Lemma~\ref{maxmin} we know $\overline{%
\lambda}_k>$ $\rho(B)$ as $\mathbf{x}_k\not=\mathbf{x}$.
Consequently, $\overline{\lambda }_{k}I-B$ is a nonsingular $M$-matrix,
and the vector $\mathbf{y}_{k+1}$ satisfies
\begin{equation*}
\mathbf{y}_{k+1}=(\overline{\lambda }_{k}I-B)^{-1}\left( \mathbf{x}_{k}+%
\mathbf{f}_{k}\right) >\mathbf{0}.
\end{equation*}
This implies $\mathbf{x}_{k+1}=\mathbf{y}_{k+1} \, / \, \Vert \mathbf{y}%
_{k+1}\Vert >\mathbf{0}$ and $\min \left( \frac{\mathbf{x}_{k}+\mathbf{f}_{k}%
}{\mathbf{y}_{k+1}}\right) >0$.

From (\ref{eq:inexactsys}) and the above
it follows that%
\begin{align}
\overline{\lambda }_{k+1}& =\max \left( \frac{B\mathbf{x}_{k+1}}{\mathbf{x}%
_{k+1}}\right) =\max \left( \frac{B\mathbf{y}_{k+1}}{\mathbf{y}_{k+1}}\right)
\notag \\
& =\max \left( \frac{\overline{\lambda }_{k}\mathbf{y}_{k+1}-\mathbf{x}_{k}-%
\mathbf{f}_{k}}{\mathbf{y}_{k+1}}\right)  \notag \\
& =\overline{\lambda }_{k}-\min \left( \frac{\mathbf{x}_{k}+\mathbf{f}_{k}}{%
\mathbf{y}_{k+1}}\right) <\overline{\lambda }_{k},  \label{eq:delamda}
\end{align}%
proving that he sequence $\{\overline{\lambda }_{k}\}$  is monotonically
decreasing. Again, by Lemma~\ref{maxmin} we have
$\overline{\lambda }_{k}>\overline{\lambda }%
_{k+1}\geq \rho (B)$.
\end{proof}

Based on (\ref{eq:inexactsys}), (\ref{eq:normal}) and (\ref{eq:delamda}) and
Lemma~\ref{monotone}, we describe INI as Algorithm~\ref{alg:iiva}.

\begin{algorithm}
\begin{enumerate}
  \item   Given an initial guess $\bx_0 > \zero$ with $\Vert \bx_0\Vert =1$,
  $0\le \gamma < 1$ and ${\sf tol}>0$, compute $\overline{\lambda}_0=
  \max \left(\frac{B\bx_0}{\bx_0}\right)$.
  \item   {\bf for} $k =0,1,2,\dots$
  \item   \quad Solve $(\overline{\lambda}_kI-B) \, \by_{k+1}= \bx_k$
approximately  by an iterative solver such that
\begin{equation*}
\phantom{MM} |(\overline{\lambda}_kI-B)\by_{k+1}-\bx_k|
= |\bff_k| \le \gamma \, \bx_k.
\end{equation*}
  \item   \quad Normalize the vector $\bx_{k+1}= \by_{k+1} \, / \,
  \Vert \by_{k+1}\Vert$.
  \item   \quad Compute $\overline{\lambda}_{k+1}= \overline{\lambda }_{k}-\min \left( \frac{\mathbf{x}_{k}+\mathbf{f}_{k}}{%
\mathbf{y}_{k+1}}\right)$.
  \item   {\bf until} convergence: $\Vert B\bx_{k+1}-\overline{\lambda}_{k+1}
  \bx_{k+1}\Vert <{\sf tol}$.
\end{enumerate}
\caption{Inexact Noda Iteration (INI)}
\label{alg:iiva}
\end{algorithm}

Note that if $\gamma =0$, i.e., $\mathbf{f}_{k}$ $=0$ in (\ref{eq:inexactsys}%
) for all $k$ then Algorithm~\ref{alg:iiva} becomes the standard NI. It
follows from Lemma~\ref{monotone} that Algorithm~\ref{alg:iiva} generates
the positive vector sequence $\{\mathbf{x}_k\}$, so it is a positivity
preserving algorithm. In what follows we will investigate convergence
conditions of INI for $\overline{\lambda }_{k}\rightarrow \rho (B)$ as $%
k\rightarrow \infty $.


\begin{lemma}
\label{coslower}Let $\mathbf{x}>\mathbf{0}$ be the unit length eigenvector
of $B$ associated with $\rho (B)$. For any vector $\mathbf{z}>\mathbf{0}$ with $%
\Vert \mathbf{z}\Vert =1$, it holds that $\cos\angle (\mathbf{z},\mathbf{x}%
)>\min (\mathbf{x})$ and
\begin{equation}
\underset{\Vert \mathbf{z}\Vert =1,\,\mathbf{z}>\mathbf{0}}{\inf }\cos
\angle (\mathbf{z},\mathbf{x})=\min (\mathbf{x}).  \label{eq:coslower}
\end{equation}
\end{lemma}

\begin{proof}
We have $\cos\angle (\mathbf{z},%
\mathbf{x})=\mathbf{z}^{T}\mathbf{x}>0$. Since $\mathbf{x}>\mathbf{0}$
and $\mathbf{z}>\mathbf{0}$ with $\Vert
\mathbf{x}\Vert =\Vert \mathbf{z}\Vert =1$, we have
$$
\mathbf{z}^T\mathbf{x}\geq \Vert \mathbf{z}\Vert_1 \min(\mathbf{x})
>\Vert \mathbf{z}\Vert \min(\mathbf{x})=\min(\mathbf{x}),
$$
where $\Vert\cdot\Vert_1$ is the vector 1-norm. So (\ref{eq:coslower})
holds.\footnote{We thank one of the referees, who suggested to us this more
direct proof than our original one.}
\end{proof}

We remark that the infimum in (\ref{eq:coslower}) cannot be attained because
$\Vert\mathbf{z}\Vert_1>\Vert\mathbf{z}\Vert=1$ strictly for any $\mathbf{z}>%
\mathbf{0}$ with $\Vert\mathbf{z}\Vert=1$, but $\Vert\mathbf{z}%
\Vert_1\rightarrow\Vert\mathbf{z}\Vert=1$ provided one of the components of $%
\mathbf{z}$ tends to one and all the others tend to zero.

By $\Vert \mathbf{x}\Vert=1$, it is easily seen that $\min (\mathbf{x})\leq
n^{-1/2}$. This upper bound is attained when all the components of $\mathbf{x%
}$ are equal to $1/\sqrt{n}$. On the other hand, we remark that Lynn and
Timlake \cite[Theorem 2.1]{lynn69} derived a compact lower bound
\begin{equation*}
\min(\mathbf{x})\geq \frac{\|\mathbf{x}\|_1\min_{i,j}b_{ij}}{\rho(B)-
\min_i\sum_j b_{ij}+n\min_{i,j}b_{ij}}>0  \label{lowerbound}
\end{equation*}
with the lower attained and equal to the upper bound $n^{-1/2}$ when $B$ is
a generalized stochastic matrix, i.e., all the row sums of $B$ are equal to
a positive constant. For such $B$, its row sum is just $\rho(B)$. So for $\|%
\mathbf{x}\|=1$, $\min(\mathbf{x})$ is always mildly small for $n$ large and
can be very small if $\min_{i,j}b_{ij}$ is very small.

Let $\{\mathbf{x}_{k}\}$ be generated by Algorithm~\ref{alg:iiva}. We
decompose $\mathbf{x}_{k}$ into the orthogonal direct sum%
\begin{equation}
\mathbf{x}_{k}=\mathbf{x} \, \cos \varphi_{k}+\mathbf{p}_{k} \, \sin \varphi
_{k},\quad \mathbf{p}_{k}\in \text{span}(V)\perp \mathbf{x}
\label{eq:decomposition}
\end{equation}%
with $\Vert \mathbf{p}_{k}\Vert =1$ and $\varphi_{k}$ $=$ $\angle \left(
\mathbf{x}_{k},\mathbf{x}\right) $ the acute angle between $\mathbf{x}_{k}$
and $\mathbf{x}$. Now define
\begin{equation}  \label{epsilonk}
\varepsilon_{k}=\overline{\lambda }_{k}-\rho (B),\quad B_{k}=\overline{%
\lambda }_{k}I-B.
\end{equation}
Then from (\ref{eqn: partition}) we have
\begin{equation*}
\left[
\begin{array}{c}
\mathbf{x}^{T} \\
V^{T}%
\end{array}%
\right] B_{k}\left[
\begin{array}{cc}
\mathbf{x} & V%
\end{array}%
\right] =\left[
\begin{array}{cc}
\varepsilon_{k} & \mathbf{c}^T \\
0 & L_{k}%
\end{array}%
\right],  \label{eqLpart2}
\end{equation*}
where $L_{k}=\overline{\lambda }_{k}I-L$. For $\overline{\lambda}%
_k\not=\rho(B)$, it is easy to verify that
\begin{equation}
\left[
\begin{array}{c}
\mathbf{x}^{T} \\
V^{T}%
\end{array}%
\right] B_{k}^{-1}\left[
\begin{array}{cc}
\mathbf{x} & V%
\end{array}%
\right] =\left[
\begin{array}{cc}
\frac{1}{\varepsilon_{k}} & \mathbf{b}_k^T \\
0 & L_{k}^{-1}%
\end{array}%
\right]
\mbox{\ \ with
$\mathbf{b}_k^T=-\frac{\mathbf{c}^TL_k^{-1}}{\varepsilon_{k}}$},
\label{eqLpart3}
\end{equation}%
from which we get
\begin{equation}  \label{BkV}
B_k^{-1}V=\mathbf{x}\mathbf{b}_k^T+VL_k^{-1}=-\mathbf{x}\frac{\mathbf{c}%
^TL_k^{-1}} {\varepsilon_{k}}+VL_k^{-1}.
\end{equation}

From Lemma~\ref{monotone}, since $\left\{ \overline{\lambda }_{k}\right\} $
is monotonically decreasing and bounded by $\rho(B)$ from below, we must
have $\lim_{k\rightarrow\infty} \overline{\lambda }_k=\alpha\geq \rho(B)$,
where $\alpha=\rho(B)$ or $\alpha>\rho(B)$. We next investigate these two
possibilities, respectively, and present some basic results that will play
an important role in proving $\alpha=\rho(B)$ when certain further
restrictions are imposed on the inner tolerance $\xi_k=\Vert \mathbf{f}%
_k\Vert$.

\begin{lemma}
\label{equiThm}Let $B$ be an irreducible nonnegative matrix,
and assume that $%
(\rho (B),\mathbf{x})$ is the largest eigenpair of $B$ with $\mathbf{x}>%
\mathbf{0}$ and $\Vert \mathbf{x}\Vert =1$. If $\mathbf{x}_{k}$, $\overline{%
\lambda }_{k} $, $\mathbf{y}_{k}$ and $\mathbf{f}_{k}$ are generated by
Algorithm~\ref{alg:iiva}, then the following statements are equivalent:%
\begin{equation*}
\text{{\rm (i)}}\underset{k\rightarrow \infty }{\lim }\mathbf{x}_{k}=\mathbf{x}%
;\quad \text{ {\rm (ii)}}\underset{k\rightarrow \infty }{\lim }\overline{\lambda }%
_{k}=\rho (B); \quad \text{{\rm (iii)}}\underset{k\rightarrow \infty }{\lim }%
\Vert \mathbf{y}_{k}\Vert ^{-1}=0.
\end{equation*}
\end{lemma}

\begin{proof}
(i)$\Rightarrow $(ii)$:$ By the definition of $\overline{\lambda }_{k}$, we
get
\begin{equation*}
\underset{k\rightarrow \infty }{\lim }\overline{\lambda }_{k}=\underset{%
k\rightarrow \infty }{\lim }\max \left( \frac{B\mathbf{x}_{k}}{\mathbf{x}_{k}%
}\right) =\max \left( \underset{k\rightarrow \infty }{\lim }\frac{B\mathbf{x}%
_{k}}{\mathbf{x}_{k}}\right) =\rho(B).
\end{equation*}

(i)$\Rightarrow $(iii)$:$ Since $|\mathbf{f}_{k}|\leq \gamma \,\mathbf{x}_{k}
$, from (\ref{eq:decomposition}) we have%
\begin{align}
\mathbf{y}_{k+1}& =B_{k}^{-1}\left( \mathbf{x}_{k}+\mathbf{f}_{k}\right)
\notag \\
& \geq B_{k}^{-1}\left( 1-\gamma \right) \mathbf{x}_{k}  \notag \\
& =\left( 1-\gamma \right) \left( \varepsilon_{k}^{-1} \, \mathbf{x} \, \cos
\varphi_{k}+B_{k}^{-1} \, \mathbf{p}_{k} \, \sin\varphi_{k}\right).
\label{eq:decomyk}
\end{align}%
Since $\mathbf{p}_{k}\in \text{span}(V)$, we can write
\begin{equation*}
\mathbf{p}_{k}=V\mathbf{d}_k
\end{equation*}
with $\|\mathbf{d}_k\|=1$. From (\ref{BkV}) and (\ref{eq:decomyk}), we get
\begin{equation}
\varepsilon_{k}^{-1} \, \mathbf{x} \, \cos \varphi_{k}+B_{k}^{-1} \, \mathbf{%
p}_{k} \, \sin\varphi_{k}= \varepsilon_{k}^{-1}(\cos\varphi_k-\mathbf{c}%
^TL_k^{-1}\mathbf{d}_k\sin\varphi_k) \mathbf{x}+VL_k^{-1}\mathbf{d}%
_k\sin\varphi_k.  \label{decomp3}
\end{equation}
From Lemma~\ref{coslower} it follows that $\cos \varphi_{k}>\min (\mathbf{x})
$ for all $k$. On the other hand, (i) means that $\cos\varphi_k\rightarrow 1$
and $\sin\varphi_k \rightarrow 0$ as $k\rightarrow \infty$. Since $\underset{%
k\rightarrow \infty }{\lim } \mathbf{x}_{k}=\mathbf{x}$, we get $%
\underset{k\rightarrow \infty }{\lim }\overline{\lambda }_{k}=\rho (B)$. As
a result, from definition (\ref{eq:sep}) and $\Vert V\Vert=1$,
the second term in the right-hand side of
(\ref{decomp3}) is bounded by%
\begin{eqnarray*}
\Vert VL_k^{-1}\mathbf{d_k}\sin\varphi_k\Vert & \leq &\|V\|\|L_k^{-1}\| \|%
\mathbf{d}_k\| \sin\varphi_k \\
&=&\frac{\sin\varphi_k}{\mathrm{sep}(\overline{\lambda }_{k},L)}\rightarrow
\frac{\sin\varphi_k}{\mathrm{sep}(\rho (B),L)}\rightarrow 0,
\label{eq:gkbound}
\end{eqnarray*}%
and
\begin{equation*}
|\mathbf{c}^TL_k^{-1}\mathbf{d}_k\sin\varphi_k|\leq \|\mathbf{c}\| \
\|L_k^{-1}\| \sin\varphi_k\rightarrow 0.
\end{equation*}

Combining (\ref{eq:decomyk}) with the above and exploiting the norm triangle
inequality, we have%
\begin{align*}
\Vert \mathbf{y}_{k+1}\Vert & \geq \left( 1-\gamma \right) \left( \left(
\overline{\lambda }_{k}-\rho (B)\right) ^{-1} \mid \cos\varphi_k-\mathbf{c}%
^TL_k^{-1}\mathbf{d}_k\sin\varphi_k\mid - \Vert VL_k^{-1}\mathbf{d}_{k}\sin
\varphi_{k}\Vert \right) \\
&\geq \left( 1-\gamma \right) \left( \frac{\left| \cos\varphi_k-| \mathbf{c}%
^TL_k^{-1}\mathbf{d}_k| \sin\varphi_k \right|}{|\overline{\lambda }_{k}-\rho
(B)|}-\frac{\sin\varphi_k}{\mathrm{sep}(\overline{\lambda }_{k},L)}\right) \\
& \geq \left( 1-\gamma \right) \left( \frac{\left|\min{(\mathbf{x})}- |%
\mathbf{c}^TL_k^{-1}\mathbf{d}_k |\sin\varphi_k \right|}{|\overline{\lambda }%
_{k}-\rho (B)|}-\frac{\sin\varphi_k}{\mathrm{sep}(\overline{\lambda }_{k},L)}%
\right) \rightarrow \infty.
\end{align*}%
%

(iii)$\Rightarrow $(ii)$:$ From $|\mathbf{f}_{k}|\leq \gamma \,\mathbf{x}$
with $0\le \gamma<1$ and
\begin{equation*}
\mathbf{y}_{k+1} =\left( \overline{\lambda }_{k}I-B\right) ^{-1}\left(
\mathbf{x}_{k}+\mathbf{f}_{k}\right),
\end{equation*}
we get
\begin{equation*}
\Vert \mathbf{y}_{k+1} \Vert \le \Vert ( \overline{\lambda }_{k}I-B)
^{-1}\Vert \left(\Vert\mathbf{x}_{k} \Vert+\Vert \mathbf{f}%
_{k}\Vert\right)\le \Vert < 2\Vert (\overline{\lambda }_{k}I-B)^{-1}\Vert,
\end{equation*}
meaning that
\begin{equation*}
\frac{1}{\|\mathbf{y}_{k+1}\|}> \frac{1}{2\| ( \overline{\lambda }%
_{k}I-B)^{-1}\|}\geq 0.
\end{equation*}
Since $\Vert \mathbf{y}_{k}\Vert^{-1}\rightarrow 0$ for $k\rightarrow \infty$%
, the sequence $\{ \overline{\lambda }_{k}I-B\}$ tends to a singular matrix,
meaning that $\{\overline{\lambda }_{k}\}$ converges to an eigenvalue of $B$%
. From Lemma~\ref{monotone}, we must have $\underset{k\rightarrow \infty }{%
\lim }\overline{\lambda }_{k}=\rho (B)$ because $\{\overline{\lambda }_{k}\}$
is monotonically decreasing with the lower bound $\rho(B)$. Otherwise,
$\underset{k\rightarrow \infty }{\lim }\overline{\lambda }_{k}=\alpha$ would be
an eigenvalue of $B$ bigger than
the spectral radius $\rho(B)$ of $B$, which is impossible.
\end{proof}


\begin{lemma}
\label{opposite}
For Algorithm~\ref{alg:iiva}, if $\lim_{k\rightarrow\infty}
\overline{\lambda }_k=\alpha >\rho (B)$, then {\rm (i)} $%
\Vert \mathbf{y}_{k}\Vert $ is bounded$;$ {\rm (ii)}
$\underset{k\rightarrow
\infty }{\lim }\min (\mathbf{x}_{k}+\mathbf{f}_{k})=0; 
 ${\rm (iii)} $\sin\angle \left( \mathbf{x},\mathbf{x}_k\right) \geq \zeta >0$
 with $\zeta$ some constant.
\end{lemma}

\begin{proof}
(i) Since $|\mathbf{f}_{k}|\leq \gamma \,\mathbf{x}_{k}$, we get
\begin{align}
\Vert \mathbf{y}_{k+1}\Vert & =\Vert \left( \overline{\lambda }%
_{k}I-B\right) ^{-1}\left( \mathbf{x}_{k}+\mathbf{f}_{k}\right) \Vert < 2
\, \Vert (\overline{\lambda }_{k}I-B)^{-1}\Vert  \notag \\
& =\frac{2}{\mathrm{sep}(\overline{\lambda }_{k},B)}\rightarrow \frac{2}
{\mathrm{sep}(\alpha ,B)}<\infty .  \label{eq:bdyk}
\end{align}%
(ii) From (\ref{eq:delamda}) it follows that
\begin{equation}
\underset{k\rightarrow \infty }{\lim }\min \left( \frac{\mathbf{x}_{k}+%
\mathbf{f}_{k}}{\mathbf{y}_{k+1}}\right) =\underset{k\rightarrow \infty}{
\lim }\left( \overline{\lambda }_{k}-\overline{\lambda }_{k+1}\right) =0.
\label{eq: minfk}
\end{equation}
On the other hand, from (\ref{eq:bdyk}) and (\ref{eq: minfk}) we have%
\begin{equation*}
\min \left( \frac{\mathbf{x}_{k}+\mathbf{f}_{k}}{\mathbf{y}_{k+1}}\right)
\geq \frac{\min (\mathbf{x}_{k}+\mathbf{f}_{k})}{\max \left( \mathbf{y}%
_{k+1}\right) }\geq \frac{\min (\mathbf{x}_{k}+\mathbf{f}_{k})}{\Vert\mathbf{y}%
_{k+1}\Vert }> \frac{\min (\mathbf{x}_{k}+\mathbf{f}_{k}) \, \mathrm{sep%
}(\overline{\lambda }_{k} ,B)}{2}>0.
\end{equation*}
Thus, we get%
\begin{equation*}
\underset{k\rightarrow \infty }{\lim }\min (\mathbf{x}_{k}+\mathbf{f}%
_{k})=0. 
\end{equation*}
(iii) Suppose there is a subsequence $\{\sin\angle (\mathbf{x},\mathbf{x}%
_{k_{j}})\}$ that converges to zero. Then from Lemma~\ref{equiThm} there
is a subsequence $\{\overline{\lambda }_{k_{j}}\}$ that converges to $\rho (B)$,
a contradiction.
\end{proof}

Keep in mind that in Lemmas~\ref{equiThm}--\ref{opposite} we only assume the
condition $| \mathbf{f}_k |\leq \gamma \mathbf{x}_k$ with $0<\gamma<1$,
under which we can only prove that the sequence $\{\overline{\lambda}_k\}$
converges to either $\rho(B)$ or $\alpha>\rho(B)$. So only this condition is
not enough to guarantee that INI computes the desired eigenpair $(\rho(B),%
\mathbf{x})$ of $B$. In order to make  $\lim_{k\rightarrow \infty}\overline{%
\lambda}_k=\rho(B)$, we have to impose some stronger conditions on $\mathbf{f%
}_k$.

\section{Convergence Analysis of two practical INI algorithms}

\label{sec:type1}

In order to make INI converge correctly and effective, we now propose the
following two practical inner tolerance strategies for the inexactness of
step 3 of Algorithm~\ref{alg:iiva}:

\begin{itemize}
\item INI\_1: the residual norm satisfies $\xi_k = \Vert \mathbf{f}_k\Vert
\le \gamma \min(\mathbf{x}_k)$ for some fixed constant $0<\gamma < 1$;

\item INI\_2: the residual vector satisfies $|\mathbf{f}_k| \le d_k\mathbf{x}%
_k$ with $d_k=1-\overline{\lambda}_k/\overline{\lambda}_{k-1}$ for $k\ge 1$
and $\Vert \mathbf{f}_0\Vert \le \gamma \min(\mathbf{x}_0)$ for some
constant $0<\gamma < 1$.
\end{itemize}

It is easily seen that the residual vectors $\mathbf{f}_k$ of INI\_1 and
INI\_2 must satisfy $|\mathbf{f}_k| \le \gamma \, \mathbf{x}_k$ with $%
0<\gamma < 1$. As a result, it is known from Lemma~\ref{monotone} that each
of INI\_1 and INI\_2 generates a monotonically decreasing sequence $\{%
\overline{\lambda}_k\} $ bounded by $\rho(B)$ and a sequence of positive
vectors $\{ \mathbf{x}_k\}$. INI\_1 and INI\_2 now require stronger
conditions on $\mathbf{f}_k$ than the previous INI.

In Sections~4.1--4.2, we will prove the global convergence of INI\_1 and
INI\_2, respectively. Furthermore, we will show that INI\_1 converges at
least linearly with the asymptotic convergence factor bounded by $\frac{%
2\gamma}{1+\gamma}$ and INI\_2 converges superlinearly with the asymptotic
convergence order $\frac{1+\sqrt{5}}{2}$. In the meantime, we revisit the
convergence of NI and derive a new quadratic convergence result, which is
included in Section 4.2.

We comment that since INI\_2 converges, we must have $\mathbf{d}%
_k\rightarrow 0$, which means that we need to solve the linear systems more
and more accurately as $k$ increases. In contrast, the inner tolerance used
by INI\_1 is fixed except the factor $\min(\mathbf{x}_k)$, which tends to $%
\min(\mathbf{x})$ as $k\rightarrow \infty$ whenever INI\_1. This means that
for a similar inner linear system, i.e., $\overline{\lambda}_k$ almost the
same, we may pay higher computational cost to solve it by INI\_2 than
INI\_1. However, note that INI\_1 and INI\_2 converges linearly and
superlinearly, respectively. Consequently, it is hard and should be
impossible to make a general theoretical comparison of their overall
efficiency. Actually, our numerical experiments will demonstrate that the
overall efficiency of INI\_1 and INI\_2 is comparable, and there is no
general winner between them.

\subsection{Linear convergence of INI\_1}

We decompose $\mathbf{x}_{k+1}$ in the same manner as (\ref{eq:decomposition}%
):
\begin{equation*}
\mathbf{x}_{k+1}=\mathbf{x}\cos\varphi_{k+1}+\mathbf{p}_{k+1}\sin%
\varphi_{k+1}, \quad \mathbf{p}_{k+1}\in \text{span}(V)\perp \mathbf{x}
\label{eq:decompo2}
\end{equation*}%
with $\Vert \mathbf{p}_{k+1}\Vert=1$. So by definition, we have $%
\cos\varphi_{k+1}=\mathbf{x}^{T}\mathbf{x}_{k+1}$ and $\sin\varphi_{k+1}=%
\Vert V^T\mathbf{x}_{k+1}\Vert$. Obviously, $\mathbf{x}_{k}\rightarrow
\mathbf{x}$ if and only if $\tan\varphi_{k}\rightarrow 0$, i.e., $%
\sin\varphi_{k}\rightarrow 0$.

From (\ref{eqLpart3}), we have
\begin{equation*}
\mathbf{x}^T B_k^{-1}=\varepsilon_k^{-1} \mathbf{x}^T-\varepsilon_k^{-1}
\mathbf{c}^TL_k^{-1}V^T.
\end{equation*}
Exploiting the above relation, $\|V^T\|=1$ and $\Vert V^T(\mathbf{x}_k+
\mathbf{f}_k)\Vert \le \sin\varphi_k+\Vert \mathbf{f}_k\Vert$, we obtain
\begin{eqnarray}
\tan\varphi_{k+1}& =&\frac{\sin\varphi_{k+1}}{\cos\varphi_{k+1}} = \frac{%
\Vert V^{T}\mathbf{x}_{k+1}\Vert}{\mathbf{x}^{T}\mathbf{x}_{k+1}} =\frac{%
\Vert V^{T}\mathbf{y}_{k+1}\Vert}{\mathbf{x}^{T}\mathbf{y}_{k+1}}  \notag \\
& =&\frac{\Vert V^{T}B_{k}^{-1}( \mathbf{x}_{k}+\mathbf{f}_{k})\Vert} {%
\mathbf{x}^{T}B_{k}^{-1}\left( \mathbf{x}_{k}+\mathbf{f}_{k}\right)}  \notag
\\
&=& \frac{\Vert L_{k}^{-1}V^{T}\left( \mathbf{x}_{k}+\mathbf{f}_{k}\right)
\Vert}{ \left( \varepsilon_k^{-1}\mathbf{x}^T-\varepsilon_k^{-1}\mathbf{c}^T
L_k^{-1}V^T \right)\left( \mathbf{x}_{k}+\mathbf{f}_{k}\right)}  \notag \\
&= &\frac{\Vert L_{k}^{-1}V^{T}\left( \mathbf{x}_{k}+\mathbf{f}_{k}\right)
\Vert}{ \varepsilon_k^{-1}\mathbf{x}^T \mathbf{x}_k-\varepsilon_k^{-1}%
\mathbf{c}^T L_k^{-1}V^T \mathbf{x}_k +\varepsilon_k^{-1}\mathbf{x}^T
\mathbf{f}_k -\varepsilon_k^{-1}\mathbf{c}^T L_k^{-1} V^T\mathbf{f}_k}
\notag \\
& \leq &\Vert L_{k}^{-1}\Vert \, \varepsilon _{k} \, \frac{\sin\varphi_k +
\Vert\mathbf{f}_k\Vert}{\cos\varphi_k-\mathbf{c}^T L_k^{-1}V^T \mathbf{x}_k
- \Vert \mathbf{f}_k \Vert -\Vert \mathbf{c}\Vert \Vert L_k^{-1}\Vert \Vert%
\mathbf{f}_k\Vert}  \notag \\
& = &\Vert L_{k}^{-1}\Vert \, \varepsilon _{k} \, \frac{\tan\varphi_k + \Vert%
\mathbf{f}_k\Vert/\cos\varphi_k}{1-\mathbf{c}^T L_k^{-1}V^T \mathbf{x}%
_k/\cos\varphi_k - (1+\Vert \mathbf{c}\Vert \Vert L_k^{-1} \Vert)\Vert%
\mathbf{f}_k\Vert/\cos\varphi_k}  \label{eq:tan}
\end{eqnarray}
with the last second inequality holding by assuming that $(1+\Vert c\Vert
\Vert L_k^{-1}) \Vert\mathbf{f}_k\Vert/\cos\varphi_k<1-\mathbf{c}^T
L_k^{-1}V^T \mathbf{x}_k/\cos\varphi_k$. This assumption must be satisfied
provided that $\Vert \mathbf{f}_{k}\Vert$ is suitably small, because, by
Lemma~\ref{coslower}, it holds that $\cos\varphi_k>\min(\mathbf{x})$ for all
$k$.

Particularly, if $\mathbf{f}_k=0$, i.e., $\gamma=0$, we recover NI and get
\begin{equation}
\tan\varphi_{k+1}\leq \frac{\Vert L_{k}^{-1}\Vert \, \varepsilon _{k}} {1-%
\mathbf{c}^T L_k^{-1}V^T \mathbf{x}_k/\cos\varphi_k} \tan\varphi_k :=\beta_k
\tan\varphi_k.  \label{niconv}
\end{equation}
We remark that
\begin{equation*}
\beta_k=\Vert L_k^{-1}\Vert \varepsilon_k
\end{equation*}
if $B$ is a normal matrix as $\mathbf{c}=0$ in this case. Since NI is
quadratically convergent \cite{Els76}, for $k$ large enough we must have
\begin{equation*}
\beta_k=O(\tan\varphi_k)\rightarrow 0.
\end{equation*}
Therefore, for any given positive constant $\beta<1$, it holds that
\begin{equation}
\beta_k<\beta<1  \label{betak}
\end{equation}
for $k\ge N$ with $N$ large enough. This means that, for $k\ge N$, we have
\begin{equation*}
\tan\varphi_{k+1}<\beta\tan\varphi_k.
\end{equation*}


\begin{theorem}
\label{main1} Let $B$ be an irreducible nonnegative matrix. If the sequence
$\{\overline{\lambda }_{k}\}$ is generated by INI\_1, then
$\{\overline{\lambda }_{k}\}$ is monotonically decreasing and
$\lim_{k\rightarrow\infty} \overline{\lambda}_{k}=\rho (B)$.
\end{theorem}

\begin{proof}
By the assumption on INI\_1, since $\xi_{k}=\Vert \mathbf{f}_{k}\Vert \leq
\gamma \min (\mathbf{x}_{k})$, it holds that $|\mathbf{f}_{k}|\leq \gamma \,%
\mathbf{x}_{k}$ with $0<\gamma<1$, which satisfies the condition in Lemma~\ref%
{monotone}. So the sequence $\left\{ \overline{\lambda }_{k}\right\}$
is bounded and monotonically decreasing, and
we must have either $\lim_{k\rightarrow\infty}\overline{\lambda }_k
=\rho(B)$ or $\lim_{k\rightarrow\infty}\overline{\lambda }_k=\alpha>\rho(B)$.
Next we prove by contradiction that,
for INI\_1, $\lim_{k\rightarrow\infty}\overline{\lambda }_k
=\rho(B)$ must hold.

Suppose that
$\lim_{k\rightarrow\infty}\overline{\lambda }%
_{k}=\alpha>\rho(B)$.
By $|\mathbf{f}_{k}|\leq \gamma \,\mathbf{x}_{k}$, we get
\begin{equation*}
\min (\mathbf{x}_{k}+\mathbf{f}_{k})\geq \left( 1-\gamma \right) \min (%
\mathbf{x}_{k}).
\end{equation*}
It follows from (ii) of Lemma~\ref{opposite} that
\begin{equation*}
0=\underset{%
k\rightarrow \infty }{\lim }\min (\mathbf{x}_{k}+\mathbf{f}_{k})\ge
\left( 1-\gamma \right)\underset{k\rightarrow \infty }{\lim } \min
(\mathbf{x}_{k})\geq 0.
\end{equation*}
Thus, we have
\begin{equation*}
\underset{k\rightarrow \infty }{\lim }\min (\mathbf{x}_{k})=0.
\label{eq:minxk0}
\end{equation*}%
From Lemma~\ref{coslower} and (iii) of Lemma~\ref{opposite} we know that $%
\sin\varphi _{k}$ and $\cos\varphi _{k}$ are uniformly bounded below by a
positive constant. Therefore, there is an $m>0$ such that
$m\leq\sin \varphi _{k}$ and $m\leq \cos\varphi _{k}$, leading to
\begin{equation}
\frac{1}{\cos\varphi_{k}}\leq \frac{1}{m}\qquad \text{and}\qquad \min (%
\mathbf{x}_{k})\leq \frac{\sin\varphi_k}{m}\min (\mathbf{x}_{k}).
\label{eq:lowersin}
\end{equation}%
Using (\ref{eq:tan}) and (\ref{eq:lowersin}), we obtain%
\begin{align*}
\tan\varphi_{k+1}& \leq \Vert L_{k}^{-1}\Vert \, \varepsilon _{k} \, \frac{%
\tan\varphi_k + \gamma\min(\mathbf{x}_k)\tan\varphi_k/m}{1-\mathbf{c}^T
L_k^{-1}V^T \mathbf{x}_k/m - \gamma(1+\Vert \mathbf{c}\Vert \Vert
L_k^{-1}\Vert) \min(\mathbf{x}_k)/m} \\
& \leq \frac{\Vert L_k^{-1}\Vert \varepsilon_k \left(1 +\gamma\min(%
\mathbf{x}_k)/m\right)}{1-\mathbf{c}^T L_k^{-1}V^T \mathbf{x}_k/m -
\gamma(1+\Vert \mathbf{c}\Vert \Vert L_k^{-1}\Vert) \min(\mathbf{x}_k)/m}
\tan\varphi_{k}.
\end{align*}%
Define
\begin{equation*}
\beta_k^{\prime}= \frac{\Vert L_k^{-1}\Vert \varepsilon_k \left(1
+\gamma\min(\mathbf{x}_k)/m\right)}{1-\mathbf{c}^T L_k^{-1}V^T \mathbf{x}%
_k/m - \gamma(1+\Vert \mathbf{c}\Vert \Vert L_k^{-1}\Vert) \min(\mathbf{x}%
_k)/m}.
\end{equation*}
Note that $\Vert L_k^{-1}\Vert\rightarrow \Vert (\alpha I-L)^{-1}\Vert$ is
uniformly bounded, and $\beta_k^{\prime}$ is a continuous
function with respect to $\min(\mathbf{x}_k)$ for $0<\gamma <1$. Then
it holds that $%
\beta_k^{\prime}\rightarrow \beta_k$ defined
by (\ref{niconv}) as $\min(\mathbf{x}_k)\rightarrow 0$.
Particularly, for $0<\gamma <1$ and any small positive number $\delta$, it
holds that $\beta_k^{\prime}\leq (1+\delta)\beta_k$ provided that $\min(%
\mathbf{x}_k)$ is suitably small. As a result, with $\beta$ defined by
(\ref{betak}), for $k\ge N$
large enough we can choose a sufficiently small $\delta$ such that
\begin{equation*}
\beta_k^{\prime}\leq (1+\delta)\beta_k\leq \beta<1
\end{equation*}
for $\min(\mathbf{x}_k)$ sufficiently small. As a result,
for $k\ge N$ with $N$ large enough and $\min(\mathbf{x}_k)$ sufficiently small,
we have
\begin{equation*}
\tan\varphi_{k+1}\le \beta\tan\varphi_k.
\end{equation*}
It then follows from this that $\tan\varphi_k\rightarrow 0$, i.e.,
$\mathbf{x}_{k}\rightarrow \mathbf{x}$. From
Lemma~\ref{monotone}, this means that $\{\overline{\lambda }_{k}\}$ converges to
$\rho (B)$ monotonically, a contradiction to the assumption that $%
\lim_{k\rightarrow\infty}\overline{\lambda }_{k}=\alpha>\rho(B)$.
\end{proof}


Theorem~\ref{main1} has proved the global convergence of INI\_1, but the
result is only qualitative and does not tell us anything on how fast INI\_1
converges. Next we precisely derive an upper bound for its asymptotic linear
convergence factor.

From (\ref{eq:delamda}) and (\ref{epsilonk}), we have
\begin{equation}
\varepsilon_{k+1}=\varepsilon_k\left(1-\min\left(\frac{\mathbf{x}_k+\mathbf{f%
}_k}{\varepsilon_k\mathbf{y}_{k+1}}\right) \right):= \varepsilon_k\rho_k
\label{eq:ratioerr}
\end{equation}%
with
\begin{equation}
\rho_k =1-\min\left(\frac{\mathbf{x}_k+\mathbf{f}_k}{\varepsilon_k\mathbf{y}%
_{k+1}}\right).  \label{rhok}
\end{equation}
Since $\overline{\lambda}_k-\overline{\lambda}_{k+1}<\overline{\lambda}%
_k-\rho(B)$, from (\ref{rhok}), (\ref{eq:delamda}) and (\ref{eq:monolam}) we
always have
\begin{equation}
\rho_k =1-\frac{\overline{\lambda}_k-\overline{\lambda}_{k+1}}{\overline{%
\lambda}_k-\rho(B)}=\frac{\overline{\lambda}_{k+1}-\rho(B)} {\overline{%
\lambda}_k-\rho(B)} <1.  \label{eq:lok}
\end{equation}


\begin{theorem}
\label{linearconv}For INI\_1, we have
$\underset{k\rightarrow \infty}{\lim} \rho_k\leq \frac{2\gamma%
}{1+\gamma}<1$, i.e., the convergence of INI\_1 is globally linear at least.
\end{theorem}

\begin{proof}
Since $\xi_{k}=\Vert \mathbf{f}_k\Vert\leq \gamma \min (\mathbf{x}_{k})$
in INI\_1, it holds that $|\mathbf{f}%
_{k}|\leq \gamma \,\mathbf{x}_{k}$. Therefore, we have
\begin{equation*}
\left( 1-\gamma \right) \mathbf{x}_{k} \leq \mathbf{x}_{k}+\mathbf{f}%
_{k}\leq \left( 1+\gamma \right) \mathbf{x}_{k}.
\end{equation*}
As $B_{k}^{-1}\geq 0$, it follows from the above relation that
\begin{equation*}
\left( 1-\gamma \right) B_{k}^{-1}\mathbf{x}_{k}\leq \mathbf{y}_{k+1}\leq
\left( 1+\gamma \right) B_{k}^{-1}\mathbf{x}_{k}.
\end{equation*}
Therefore, we have%
\begin{equation*}
\min \left( \frac{\mathbf{x}_{k}+\mathbf{f}_{k}}{\varepsilon_{k}\mathbf{y}%
_{k+1}}\right) \geq \min \left(\frac{(1-\gamma)\mathbf{x}_{k}}{(1+\gamma)
\varepsilon_{k}B_k^{-1}\mathbf{x}_{k}}\right)=\frac{1-\gamma}{1+\gamma}
\min\left(\frac{\mathbf{x}_{k}}{\varepsilon_{k}B_k^{-1}\mathbf{x}_{k}}%
\right).  \label{eq:linearconv}
\end{equation*}%
From (\ref{Bkinv}), we get
\begin{equation}
\varepsilon_k B_k^{-1}\mathbf{x}_k=\mathbf{x}\mathbf{x}^T \mathbf{x}_k-\mathbf{x}
\mathbf{c}^T L_k^{-1}\mathbf{V}^T\mathbf{x}_k+\varepsilon_k V L_k^{-1} V^T\mathbf{x}_k.
\label{rela3}
\end{equation}
Since, from Theorem~\ref{main1}, $\underset{k\rightarrow \infty }{\lim }
\mathbf{x}_{k}=\mathbf{x}$ and $\underset{k\rightarrow \infty }{\lim }
\overline{\lambda}_k=\rho(B)$, we have $\varepsilon_k\rightarrow 0$ and
$L_k^{-1}\rightarrow \left(\rho(B) I-L\right)^{-1}$. On the other
hand, since $L_k^{-1}\rightarrow (\rho(B) I-L)^{-1}$ and
$\underset{k\rightarrow \infty }{\lim }V^T\mathbf{x}_k=
V^T\mathbf{x}=0$, from (\ref{rela3}) we get
\begin{equation*}
\underset{k\rightarrow \infty }{\lim }\varepsilon_{k}B_{k}^{-1}\mathbf{x}_{k}
=\mathbf{x}.
\end{equation*}
Consequently, we obtain
\begin{align*}
\underset{k\rightarrow \infty }{\lim }\min \left( \frac{\mathbf{x}_{k}+%
\mathbf{f}_{k}}{\varepsilon_{k}\mathbf{y}_{k+1}}\right) & \geq \frac{1-\gamma%
}{1+\gamma} \min\left(\underset{k\rightarrow \infty }{\lim}\frac{\mathbf{x}%
_{k}}{\varepsilon_{k} B_k^{-1}\mathbf{x}_{k}}\right) \\
&= \frac{1-\gamma}{1+\gamma} \min\left(\frac{\mathbf{x}}{\mathbf{x}}\right) =%
\frac{1-\gamma}{1+\gamma}>0,
\end{align*}%
leading to
\begin{equation}
\underset{k\rightarrow \infty }{\lim }\rho_{k}\leq 1-\frac{1-\gamma }{
1+\gamma }=\frac{2\gamma}{1+\gamma}<1.  \label{convfactor}
\end{equation}
\end{proof}

It is seen from (\ref{convfactor}) that if $\gamma$ is small then INI\_1
must ultimately converge fast.

\subsection{Superlinear convergence of INI\_2}

In this subsection, we establish the global convergence theory of INI\_2 and
prove that its superlinear convergence order is $\frac{1+\sqrt{5}}{2}$. In
addition, we derive a relationship between the eigenvalue error $%
\varepsilon_k= \overline{\lambda}_k-\rho(B)$ and the eigenvector error $%
\tan\varphi_k$, which holds for both NI and INI. In the meantime, as an
important complement, we revisit the convergence of NI and prove its
quadratic convergence in terms of $\tan\varphi_k$.

\begin{theorem}
\label{main2}Let $B$ be an irreducible nonnegative matrix. If $\overline{%
\lambda }_{k}$ is generated by INI\_2, then $\underset{k\rightarrow \infty }
{\lim }\overline{\lambda }_{k}=\rho (B)$.
\end{theorem}

\begin{proof}
Since $|\mathbf{f}_{k}|\leq d_{k}\mathbf{x}_{k}$ with $d_{k}=\left(\overline{%
\lambda }_{k-1}-\overline{\lambda }_{k}\right)/\overline{\lambda }_{k-1}<1$
and $\Vert \mathbf{f}_0\Vert \le \gamma \min(\mathbf{x}_0)$ with
$0<\gamma < 1$, $\mathbf{f}_k$ satisfies the condition of Lemma~\ref%
{monotone}. Therefore, the sequence $\{\overline{\lambda}_k\}$ generated by
INI\_2 is monotonically decreasing with the lower bound $\rho(B)$ and
converges to some limit $\alpha\ge \rho(B)$. This shows
that $\xi_{k}=\Vert \mathbf{f}_{k}\Vert \rightarrow 0$ as $k\rightarrow
\infty $. Assume that $\underset{k\rightarrow \infty }{\lim }\overline{%
\lambda }_{k}=\alpha >\rho(B)$. Then there exists a positive integer
$N_1$ such that $\xi_k\leq \gamma<1$ with $\gamma$ defined in INI\_1.
For $\delta$, $\beta_k^{\prime}$ and $N$ in the proof of
Theorem~\ref{main1}, take $N_2=\max\{N_1,N\}$. Then for $k\geq N_2$ we get
\begin{equation*}
\tan\varphi_{k+1}\leq \beta \tan\varphi_k.
\end{equation*}
Hence, it holds that $\underset{k\rightarrow \infty }{\lim }\overline{%
\lambda }_{k}=\rho (B)$.
\end{proof}

We further have the following result.

\begin{theorem}
\label{eslonyk} For INI\_2, it holds that
\begin{equation*}
\underset{k\rightarrow \infty}{\lim}\varepsilon_k\mathbf{y}_{k+1} = \mathbf{x%
},
\end{equation*}
where $\varepsilon_k$ is defined by {\rm (\ref{epsilonk})}.
\end{theorem}

\begin{proof}
From (\ref{eqLpart3}), we have
\begin{equation}
B_{k}^{-1}=\frac{1}{\varepsilon_{k}}\mathbf{x}\mathbf{x}^{T}- \frac{\mathbf{x%
}\mathbf{c}^T L_k^{-1} V^T}{\varepsilon_k} +VL_k^{-1}V^{T}.
\label{Bkinv}
\end{equation}
Therefore, we get
\begin{align*}
\varepsilon_{k}\mathbf{y}_{k+1}& =\varepsilon_{k}B_{k}^{-1}\left( \mathbf{x}%
_{k}+\mathbf{f}_{k}\right)  \notag \\
& =\left( \mathbf{x}\mathbf{x}^{T}-\mathbf{x}\mathbf{c}^T L_k^{-1} V^T
+\varepsilon_kVL_k^{-1}V^{T} \right) \left( \mathbf{x}_{k}+\mathbf{f}%
_{k}\right) .  \label{eq:eslonyk}
\end{align*}%
Since $L_k^{-1}\rightarrow \left(\rho(B) I-L\right)^{-1}$ and $\varepsilon_{k}
\rightarrow 0$, we have $%
\varepsilon_{k}VL_{k}^{-1} V^T\rightarrow 0$, from which it follows that
\begin{equation}
\underset{k\rightarrow \infty }{\lim }\varepsilon_{k}\Vert
VL_{k}^{-1}V^{T}\left( \mathbf{x}_{k}+\mathbf{f}_{k}\right) \Vert =0.
\label{rela1}
\end{equation}
From Lemma~\ref{equiThm} and the proof of Lemma~\ref{main2},
we know that $\mathbf{x}_k\rightarrow \mathbf{x}$ and $\mathbf{f}_k\rightarrow 0$,
which lead to $\underset{k\rightarrow
\infty }{\lim }(\mathbf{x}_{k}+\mathbf{f}_{k})=\mathbf{x}$. Note that
$V^T\mathbf{x}=0$.
We then get
\begin{equation}
\lim_{k\rightarrow \infty}\mathbf{x}\mathbf{c}^T L_k^{-1} V^T \left( \mathbf{%
x}_{k}+\mathbf{f}_{k}\right) = \mathbf{x}\mathbf{c}^T \left(\rho(B)
I-L\right)^{-1} V^T\mathbf{x}=0.
\label{rela2}
\end{equation}
A combination of (\ref{rela1}) and (\ref{rela2}) shows that
\begin{equation*}
\lim_{k\rightarrow \infty}\varepsilon_{k}\mathbf{y}_{k+1}
=\lim_{k\rightarrow \infty}\left( \mathbf{x}\mathbf{x}^{T}- \mathbf{x}%
\mathbf{c}^T L_k^{-1} V^T +\varepsilon_kVL_k^{-1}V^{T} \right) \left(
\mathbf{x}_{k}+\mathbf{f}_{k}\right)=\mathbf{x}\mathbf{x}^{T} \mathbf{x}=%
\mathbf{x}.
\end{equation*}
\end{proof}

\begin{theorem}
\label{superconv} Define the residual $\mathbf{r}_k=(\overline{\lambda }_{k} I-B)
\mathbf{x}_k$. Then for INI\_2, the following results hold:
\begin{equation*}
\text{{\rm (i)}}\underset{k\rightarrow \infty }{\lim }\frac{\varepsilon_{k+1}}{%
\varepsilon_{k}}=0;\text{ {\rm (ii)}}\underset{k\rightarrow \infty }{\lim }\frac{%
\overline{\lambda }_{k}-\overline{\lambda }_{k+1}}{\overline{\lambda }_{k-1}-%
\overline{\lambda }_{k}}=0;\text{{\rm (iii)}}\underset{k\rightarrow \infty }{\lim
}\frac{\left\Vert \mathbf{r}_{k+1}\right\Vert }{\left\Vert \mathbf{r}%
_{k}\right\Vert }=0,
\end{equation*}%
that is, the convergence of INI\_2 is superlinear.
\end{theorem}

\begin{proof}
(i): Theorem~\ref{main2} has proved that
 $\underset{k\rightarrow \infty }{\lim }\overline{\lambda}_k=\rho(B)$
 and $\underset{k\rightarrow \infty }{\lim }\mathbf{x}_k=\mathbf{x}$.
Recall from the proof of Theorem~\ref{main2} that $\xi_{k}=\Vert \mathbf{f}%
_{k} \Vert \rightarrow 0$. Then it follows from (\ref%
{eq:ratioerr}) and Theorems~\ref{main2}--\ref{eslonyk} that
\begin{align}
\underset{k\rightarrow \infty }{\lim }\rho_k= \underset{k\rightarrow \infty }%
{\lim }\frac{\varepsilon_{k+1}}{\varepsilon _{k}}& =1-\underset{k\rightarrow
\infty }{\lim }\min \left( \frac{\mathbf{x}_{k}+\mathbf{f}_{k}}{%
\varepsilon_{k}\mathbf{y}_{k+1}}\right)  \notag \\
& =1-\min \left( \underset{k\rightarrow \infty }{\lim }\frac{\mathbf{x}_{k}+%
\mathbf{f}_{k}}{\varepsilon_{k}\mathbf{y}_{k+1}}\right)  \notag \\
& =1-\min \left( \underset{k\rightarrow \infty }{\lim }\frac{\mathbf{x}}{%
\mathbf{x}}\right) =0.  \label{eq:(i)}
\end{align}%
(ii): From (\ref{eq:ratioerr}) and (\ref{eq:(i)}), we have%
\begin{align*}
\underset{k\rightarrow \infty }{\lim }\frac{\overline{\lambda }_{k}-%
\overline{\lambda }_{k+1}}{\overline{\lambda }_{k-1}-\overline{\lambda }_{k}}%
& =\underset{k\rightarrow \infty }{\lim }\frac{\varepsilon_{k}-\varepsilon
_{k+1}}{\varepsilon_{k-1}-\varepsilon_{k}}=\underset{k\rightarrow \infty }{%
\lim }\frac{\varepsilon_{k}\left( 1-\rho_{k}\right) }{\varepsilon
_{k-1}\left( 1-\rho_{k-1}\right) } \\
& =\underset{k\rightarrow \infty }{\lim }\frac{\left( 1-\rho_{k}\right)
\rho_{k-1}}{\left( 1-\rho_{k-1}\right) }=0.
\end{align*}%
(iii): From (\ref{eq:inexactsys}), we have%
\begin{equation*}
(\overline{\lambda }_{k-1}-B)\,\mathbf{y}_{k}=\mathbf{x}_{k-1}+\mathbf{f}%
_{k-1},
\end{equation*}%
from which it follows that%
\begin{equation*}
B\mathbf{x}_{k}=\overline{\lambda }_{k-1}\left( \,\mathbf{y}_{k}-%
\mathbf{x}_{k-1}-\mathbf{f}_{k-1}\right) /\left\Vert \mathbf{y}%
_{k}\right\Vert .
\end{equation*}%
Therefore, we get
\begin{eqnarray}
\left\Vert \mathbf{r}_{k}\right\Vert  &=&\left\Vert (\overline{\lambda }%
_{k}-B)\mathbf{x}_{k}\right\Vert =\frac{\left\Vert \overline{\lambda }_{k}\mathbf{y}_{k}%
-\overline{\lambda }_{k-1}\left( \,\mathbf{y}_{k}-\mathbf{x}_{k-1}-\mathbf{f}_{k-1}\right)
\right\Vert }{\left\Vert \mathbf{y}_{k}\right\Vert } \notag \\
&=&\frac{\varepsilon _{k-1}\left\Vert \left( \overline{\lambda }_{k}-\overline{%
\lambda }_{k-1}\right) \,\mathbf{y}_{k}-\overline{\lambda }_{k-1}\left(%
\mathbf{x}_{k-1}+\mathbf{f}_{k-1}\right) \right\Vert }{\left\Vert \varepsilon _{k-1}%
\mathbf{y}_{k}\right\Vert }. \label{eq:(resi)}
\end{eqnarray}%
From (\ref{eq:lok}), we have
\begin{eqnarray}
\varepsilon _{k}\left( \overline{\lambda }_{k-1}-\overline{\lambda }%
_{k}\right) ^{-1} &=&\frac{\overline{\lambda }_{k}-\rho (B)}{\overline{%
\lambda }_{k-1}-\overline{\lambda }_{k}} \notag \\
&=&\frac{\rho _{k-1}}{1-\rho _{k-1}}. \label{eq:(epslamd)}
\end{eqnarray}
Then it follows from Theorem~\ref{eslonyk}, (ii) of Theorem~\ref{superconv}
and (\ref{eq:(i)})--(\ref{eq:(epslamd)}) that
\begin{eqnarray*}
\underset{k\rightarrow \infty }{\lim }\frac{\left\Vert \mathbf{r}%
_{k+1}\right\Vert }{\left\Vert \mathbf{r}_{k}\right\Vert } &=&\underset{%
k\rightarrow \infty }{\lim }\frac{\varepsilon _{k}\left\Vert \left(
\overline{\lambda }_{k+1}-\overline{\lambda }_{k}\right) \,\mathbf{y}_{k+1}-%
\overline{\lambda }_{k}\left( \mathbf{x}_{k}+\mathbf{f}_{k}\right)
\right\Vert \left\Vert \varepsilon _{k-1}\mathbf{y}_{k}\right\Vert }{%
\varepsilon _{k-1}\left\Vert \left( \overline{\lambda }_{k}-\overline{\lambda }%
_{k-1}\right) \,\mathbf{y}_{k}-\overline{\lambda }_{k-1}\left( \mathbf{x}%
_{k-1}+\mathbf{f}_{k-1}\right) \right\Vert \left\Vert \varepsilon _{k}%
\mathbf{y}_{k+1}\right\Vert } \\
&=&\underset{k\rightarrow \infty }{\lim }\frac{\varepsilon _{k}\left\Vert
\left( \overline{\lambda }_{k}-\overline{\lambda }_{k+1}\right) \,\mathbf{y}%
_{k+1}+\overline{\lambda }_{k}\left( \mathbf{x}_{k}+\mathbf{f}_{k}\right)
\right\Vert \left\Vert \varepsilon _{k-1}\mathbf{y}_{k}\right\Vert }{%
\varepsilon _{k-1}\left\Vert \left( \overline{\lambda }_{k-1}-\overline{%
\lambda }_{k}\right) \,\mathbf{y}_{k}+\overline{\lambda }_{k-1}\left(
\mathbf{x}_{k-1}+\mathbf{f}_{k-1}\right) \right\Vert \left\Vert \varepsilon
_{k}\mathbf{y}_{k+1}\right\Vert } \\
&=&\underset{k\rightarrow \infty }{\lim }\frac{\left\Vert \frac{\overline{%
\lambda }_{k}-\overline{\lambda }_{k+1}}{\overline{\lambda }_{k-1}-\overline{%
\lambda }_{k}}\varepsilon _{k}\,\mathbf{y}_{k+1}+\overline{\lambda }%
_{k}\left( \mathbf{x}_{k}+\mathbf{f}_{k}\right) \varepsilon _{k}\left(
\overline{\lambda }_{k-1}-\overline{\lambda }_{k}\right) ^{-1}\right\Vert
\left\Vert \varepsilon _{k-1}\mathbf{y}_{k}\right\Vert }{\left\Vert
\varepsilon _{k-1}\,\mathbf{y}_{k}+\overline{\lambda }_{k-1}\left( \mathbf{x}%
_{k-1}+\mathbf{f}_{k-1}\right) \varepsilon _{k-1}\left( \overline{\lambda }%
_{k-1}-\overline{\lambda }_{k}\right) ^{-1}\right\Vert \left\Vert
\varepsilon _{k}\mathbf{y}_{k+1}\right\Vert } \\
&=&\underset{k\rightarrow \infty }{\lim }\frac{\left\Vert \frac{\overline{%
\lambda }_{k}-\overline{\lambda }_{k+1}}{\overline{\lambda }_{k-1}-\overline{%
\lambda }_{k}}\varepsilon _{k}\,\mathbf{y}_{k+1}+\overline{\lambda }%
_{k}\left( \mathbf{x}_{k}+\mathbf{f}_{k}\right) \left( \frac{\rho
_{k-1}}{1-\rho _{k-1}}\right) \right\Vert \left\Vert \varepsilon _{k-1}\mathbf{%
y}_{k}\right\Vert }{\left\Vert \varepsilon _{k-1}\,\mathbf{y}_{k}+\overline{%
\lambda }_{k-1}\left( \mathbf{x}_{k-1}+\mathbf{f}_{k-1}\right) \left( \frac{%
1}{1-\rho _{k-1}}\right) \right\Vert \left\Vert \varepsilon _{k}%
\mathbf{y}_{k+1}\right\Vert } \\
&=&\underset{k\rightarrow \infty }{\lim }\frac{\left\Vert \frac{\overline{%
\lambda }_{k}-\overline{\lambda }_{k+1}}{\overline{\lambda }_{k-1}-\overline{%
\lambda }_{k}}\mathbf{x}+0\right\Vert \left\Vert \mathbf{x}\right\Vert }{%
\left\Vert \mathbf{x}+\rho(B)\mathbf{x} \right\Vert \left\Vert \mathbf{x}\right\Vert }=0.
\end{eqnarray*}%
\end{proof}

Although Theorem~\ref{superconv} has established the superlinear convergence
of INI\_2, it does not reveal the convergence order. Our next concern is to
derive the precise convergence order of INI\_2. This is more informative and
instructive to understand how fast INI\_2 converges.

Elsner \cite{Els76} proved the (asymptotic) quadratic convergence of the
sequence $\{\overline{\lambda}_k-\rho(B)\}$, but the constant factor
(multiplier) in his quadratic convergence result appears hard to quantify or
estimate. Below we establish an intimate and quantitative relationship
between the eigenvalue error $\varepsilon_k= \overline{\lambda}_k-\rho(B)$
and the eigenvector error $\tan\varphi_k$. This result plays a crucial role
in deriving the precise convergence order of INI\_2 and proving the
quadratic convergence of NI in terms of $\tan\varphi_k$, with the constant
factor in the quadratic convergence result given explicitly.

\begin{theorem}
\label{newquad} For NI, INI\_1 and INI\_2, we have
\begin{equation}
\varepsilon_k\leq \frac{2\Vert B\Vert}{\min(\mathbf{x})}
\tan\varphi_k+O(\tan^2\varphi_k)  \label{evaerror}
\end{equation}
for $k$ large enough. For NI we have
\begin{equation}
\tan\varphi_{k+1}\leq \frac{2\Vert B\Vert}{\min(\mathbf{x}) \mathrm{sep}(%
\overline{\lambda}_k,L)}\tan^2\varphi_k+O(\tan^3\varphi_k)
\label{evecerror}
\end{equation}
for $k$ large enough, that is, asymptotically, NI converges quadratically.
\end{theorem}

\begin{proof}
Since NI, INI\_1 and INI\_2 converge, for $k$ large enough we must have
\begin{equation*}
\vert \mathbf{p}_k\vert\tan\varphi_k\ll \mathbf{x}, \mbox{ i.e., }  \vert
\mathbf{p}_k\vert\sin\varphi_k\ll \mathbf{x}\cos\varphi_k.
\end{equation*}
Therefore, from (\ref{eq:decomposition}), the nonnegativity of $B$ and $\Vert%
\mathbf{p}_k\Vert=1$, for $k$ large enough we obtain
\allowdisplaybreaks{\begin{eqnarray*}
\overline{\lambda }_{k}&=& \max\left(\frac{B\mathbf{x}_k}{\mathbf{x}_k}%
\right) =\max\left(\frac{B(\mathbf{x}\cos\varphi_k+\mathbf{p}_k\sin\varphi_k)%
} {\mathbf{x}\cos\varphi_k+\mathbf{p}_k\sin\varphi_k}\right) \\
&=&\max\left(\frac{\rho(B)\mathbf{x}\cos\varphi_k+B\mathbf{p}_k\sin\varphi_k%
} {\mathbf{x}\cos\varphi_k+\mathbf{p}_k\sin\varphi_k}\right) \\
&\leq&\max\left(\frac{\rho(B)\mathbf{x}\cos\varphi_k+B\mathbf{p}%
_k\sin\varphi_k} {\mathbf{x}\cos\varphi_k-\vert\mathbf{p}_k\vert\sin\varphi_k%
}\right) \\
&=&\max\left(\frac{\rho(B)\mathbf{x}+B\mathbf{p}_k\tan\varphi_k} {\mathbf{x}%
-\vert\mathbf{p}_k\vert\tan\varphi_k}\right) \\
&=&\max\left(\frac{\rho(B)\left(\mathbf{x}-\vert\mathbf{p}%
_k\vert\tan\varphi_k\right) +\rho(B)\vert\mathbf{p}_k\vert\tan\varphi_k+B%
\mathbf{p}_k\tan\varphi_k} {\mathbf{x}-\vert\mathbf{p}_k\vert\tan\varphi_k}%
\right) \\
&\leq&\max\left(\frac{\rho(B)\left(\mathbf{x}-\vert\mathbf{p}%
_k\vert\tan\varphi_k\right) +\rho(B)\vert\mathbf{p}_k\vert\tan\varphi_k+B%
\vert\mathbf{p}_k\vert\tan\varphi_k} {\mathbf{x}-\vert\mathbf{p}%
_k\vert\tan\varphi_k}\right) \\
&\leq&\max\left(\frac{\rho(B)\left(\mathbf{x}-\vert\mathbf{p}%
_k\vert\tan\varphi_k\right) }{\mathbf{x}-\vert\mathbf{p}_k\vert\tan\varphi_k}%
\right) +\tan\varphi_k\max\left(\frac{\rho(B)\vert\mathbf{p}_k\vert+B\vert%
\mathbf{p}_k\vert} {\mathbf{x}-\vert\mathbf{p}_k\vert\tan\varphi_k}\right) \\
&\leq&\rho(B)+\frac{\rho(B)+\Vert B\Vert}{\min(\mathbf{x})}
\tan\varphi_k+O(\tan^2\varphi_k) \\
&\leq&\rho(B)+\frac{2\Vert B\Vert}{\min(\mathbf{x})}\tan\varphi_k+
O(\tan^2\varphi_k).
\end{eqnarray*}}
Therefore, we get
\begin{equation*}
\varepsilon_k=\overline{\lambda }_{k}-\rho(B)\leq \frac{2\Vert B\Vert}{\min(\mathbf{x})}
\tan\varphi_k+O(\tan^2\varphi_k).
\end{equation*}
Since $\Vert V^T\mathbf{x}_k\Vert =\sin\varphi_k$, we have
\begin{equation*}
\vert\mathbf{c}^TL_k^{-1}V^T \mathbf{x}_k/\cos\varphi_k\vert \leq \Vert%
\mathbf{c} \Vert \Vert L_k^{-1}\Vert \tan\varphi_k\rightarrow 0
\end{equation*}
as $k$ increases. Note that $\mathrm{sep}(\overline{\lambda}_k,L)=\frac{1} {%
\Vert L_k^{-1}\Vert}$. Then from (\ref{niconv}) we obtain
\begin{eqnarray*}
\tan\varphi_{k+1} &\leq &\frac{\varepsilon_k}{\mathrm{sep}
(\overline{\lambda}_k,L)}\frac{1}{1-\Vert\mathbf{c} \Vert \Vert
L_k^{-1}\Vert \tan\varphi_k} \tan\varphi_k \\
&=&\frac{\varepsilon_k}{\mathrm{sep} (\overline{\lambda}%
_k,L)}\left(1+\Vert\mathbf{c} \Vert \Vert L_k^{-1}\Vert \tan\varphi_k
+O(\tan^2\varphi_k)\right) \tan\varphi_k
\end{eqnarray*}
for $k$ large enough,
from which and (\ref{evaerror}) it follows that (\ref{evecerror}) holds.

Since $\mathrm{sep}(\overline{\lambda }_{k},L)\rightarrow \mathrm{sep}%
(\rho(B),L)$, (\ref{evecerror}) proves the (asymptotic) quadratic
convergence of NI.
\end{proof}


\begin{theorem}
\label{quadconvini} For $k$ large enough, the inner tolerance $\xi_k$ in
INI\_2 satisfies
\begin{equation}
\xi_k=\Vert\mathbf{f}_k\Vert=O(\tan\varphi_{k-1}), \label{innertol}
\end{equation}
and INI\_2 converges superlinearly in
the form of
\begin{equation}
\tan\varphi_{k+1}\leq C\tan^{\alpha}\varphi_k  \label{convorder}
\end{equation}
with the convergence order $\alpha=\frac{1+\sqrt{5}}{2}\approx 1.618$ and $C$
a constant.
\end{theorem}

\begin{proof}
By the condition of INI\_2 and $\overline{\lambda}_{k-1}> \overline{\lambda}%
_k \ge \rho(B)$, we have
\begin{equation*}
\xi_k=\Vert \mathbf{f}_k\Vert\leq d_k=\frac{\overline{\lambda}_{k-1}-\overline{%
\lambda}_k} {\overline{\lambda}_{k-1}}\leq \frac{\overline{\lambda}%
_{k-1}-\rho(B)} {\overline{\lambda}_{k-1}}\leq \frac{\overline{\lambda}%
_{k-1}-\rho(B)} {\rho(B)}.
\end{equation*}
Note that (\ref{evaerror}) holds for INI. Therefore, for $k$ large enough we have
\begin{equation*}
\Vert\mathbf{f}_k\Vert=O(\tan\varphi_{k-1}),
\end{equation*}
which is just (\ref{innertol}).

It is known from Lemma~\ref{coslower} that $\cos\varphi_k>\min(\mathbf{x})$
for all $k$. For $k$ large enough, note that $\mathrm{sep}(\overline{\lambda}_k,L)
\rightarrow \mathrm{sep}(\rho(B),L)$. Make the Taylor expansion of
the reciprocal of the denominator in (\ref{eq:tan}). Then
substituting (\ref{evaerror}) into (\ref{eq:tan}) and amplifying the term
\begin{equation*}
\vert\mathbf{c}^TL_k^{-1}V^T \mathbf{x}_k/\cos\varphi_k\vert \leq \Vert%
\mathbf{c} \Vert \Vert L_k^{-1}\Vert \tan\varphi_k,
\end{equation*}
 by some elementary manipulation we get
\begin{equation*}
\tan\varphi_{k+1}\leq C_1\tan^2\varphi_k+C_2\tan\varphi_k\tan\varphi_{k-1}
\end{equation*}
with $C_1$ and $C_2$ certain positive constants. Since $\tan\varphi_k<\tan%
\varphi_{k-1}$ for $k$ large enough, the above inequality can be written as
\begin{equation*}
\tan\varphi_{k+1}\leq C\tan\varphi_k\tan\varphi_{k-1}
\end{equation*}
with $C$ a positive constant. Taking the equality sign in the above
relation, by the theory of linear difference equation \cite[p. 436-7]{wilk65},
for $k$ sufficiently large we obtain
\begin{equation*}
\tan\varphi_{k+1}\leq C^{\alpha-1}\tan^\alpha\varphi_{k}:=
C\tan^\alpha\varphi_{k}
\end{equation*}
with $\alpha=\frac{1+\sqrt{5}}{2}$ and the final $C:=C^{\alpha-1}$,
which proves (\ref{convorder}).
\end{proof}

We comment that if $\mathbf{f}_k=\mathbf{0}$ then $C_2=0$ in the above
proof, in which case INI becomes NI and (\ref{evecerror}), the quadratic
convergence of NI, is recovered. (\ref{innertol}) indicates that the inner
tolerance $\Vert \mathbf{f}_k\Vert $ in INI\_2 decreases like $%
\tan\varphi_{k-1}$ with increasing $k$, so we may need to solve the inner
linear systems more and more accurately as iterations proceed. As a
compensation and gain, however, since INI\_1 converges linearly, INI\_2 may
use fewer outer iterations to achieve the convergence than INI\_1. A
consequence is that it is hard and even impossible to compare the overall
efficiency of INI\_1 and INI\_2 and draw a general definitive conclusion on
which one of them is more efficient.

\section{Computing the smallest eigenpair of an irreducible nonsingular $M$%
-matrix}

\label{sec:Mmatrix}In this section, we consider how to compute the smallest
eigenpair of an irreducible nonsingular $M$-matrix $A$. In order to propose
INI for this kind of problem, suppose that $A$ is expressed as $A=\sigma
I-B\ $, and let $\left(\lambda, \mathbf{x}\right)$ be the smallest eigenpair
of it. As have been proved previously, each of INI\_1 and INI\_2 generates a
monotonically decreasing sequence $\{\overline{\lambda}_k\}$ that converges
to $\rho(B)$ with $\overline{\lambda}_k> \rho(B)$. We denote $\underline{%
\lambda}_k =\sigma-\overline{\lambda}_k$. It follows that $\underline{\lambda%
}_k<\lambda $ and $\{\underline{\lambda}_k\}$ forms a monotonically
increasing sequence that converges to $\lambda$. From Algorithm \ref%
{alg:iiva}, at iteration $k$ the approximate solution $\mathbf{y}_{k+1}$
exactly solves the linear system%
\begin{equation}
(\overline{\lambda}_kI-B) \, \mathbf{y}_{k+1}= \mathbf{x}_k+\mathbf{f}_k,
\label{eq:linearB}
\end{equation}%
where $\overline{\lambda}_k = \max \left(\frac{B\mathbf{x}_k}{\mathbf{x}_k}%
\right)$. We then update the next iterate as $\mathbf{x}_{k+1}= \mathbf{y}%
_{k+1}/\Vert \mathbf{y}_{k+1}\Vert$. Since%
\begin{equation*}
\overline{\lambda}_kI-B=\left(\sigma I-B\right)+ (\overline{\lambda}%
_k-\sigma) I =A-\underline{\lambda}_kI,
\end{equation*}
(\ref{eq:linearB}) is equivalent to
\begin{equation*}
\left(A-\underline{\lambda}_kI\right) \mathbf{y}_{k+1}= \mathbf{x}_k+\mathbf{%
f}_k,
\end{equation*}
where%
\begin{equation*}
\underline{\lambda}_k =\sigma-\max \left(\frac{B\mathbf{x}_k}{\mathbf{x}_k}%
\right) = \min\left(\frac{A\mathbf{x}_k}{\mathbf{x}_k}\right).
\end{equation*}
Since $\left(A-\underline{\lambda}_kI\right)$ is an irreducible nonsingular $%
M$-matrix, we have $\mathbf{y}_{k+1}> \mathbf{0}$ provided that $\mathbf{x}%
_k+\mathbf{f}_k>0$. Thus, we get the relation
\begin{equation*}
\underline{\lambda}_{k+1}= \min\left(\frac{A\mathbf{x}_{k+1}}{\mathbf{x}%
_{k+1}}\right) =\underline{\lambda}_k+\min\left(\frac{\mathbf{x}_k+\mathbf{f}%
_k}{\mathbf{y}_{k+1}}\right).  \label{updatem}
\end{equation*}

Therefore, Algorithm~\ref{alg:iiva} can be adapted to computing the smallest
eigenpair of the irreducible nonsingular $M$-matrix $A$, which is described
as Algorithm~\ref{alg:ivaM}, where, as in Section \ref{sec:type1}, we define

\begin{itemize}
\item INI\_1: the residual norm satisfies $\xi_k\le \gamma \min(\mathbf{x}_k)
$ for some $0<\gamma < 1$.

\item INI\_2: the residual vector satisfies $|\mathbf{f}_k| \le d_k\mathbf{x}%
_k$ with $d_k = 1-\underline{\lambda}_{k-1}/\underline{\lambda}_k$ for $k
\ge 1$ and $\Vert \mathbf{f}_0\Vert \le \gamma \min(\mathbf{x}_0)$ with some
$0<\gamma< 1$.
\end{itemize}

We should point out that Algorithm~\ref{alg:ivaM} itself neither involves $%
\sigma$ nor requires that $A$ be expressed as $A=\sigma I-B$, which is
purely for the algorithmic derivation.

\begin{algorithm}
\begin{enumerate}
  \item Given an initial guess $\bx_0> \zero$ with $\|\bx_0\|=1$ and ${\sf tol}>0$,
  compute $\underline{\lambda}_0
  = \min\left(\frac{A\bx_0}{\bx_0}\right)$.
  \item {\bf for} $k =0,1,2,\dots$
  \item \quad Solve $\left(A-\underline{\lambda}_kI\right)
  \by_{k+1}= \bx_k$ approximately with the first or second inner
  tolerance strategy such that
  $$
  \left(A-\underline{\lambda}_kI\right)
  \by_{k+1}= \bx_k+ \mathbf{f}_k.
  $$
  \item \quad Normalize the vector $\bx_{k+1}= \by_{k+1}/\Vert \by_{k+1}\Vert$.
  \item \quad Compute $\underline{\lambda}_{k+1}=\underline{\lambda}_k+
  \min\left(\frac{\mathbf{x}_k+\mathbf{f}%
_k}{\mathbf{y}_{k+1}}\right)$.
  \item {\bf until} convergence: $\Vert A\bx_{k+1}-\underline{\lambda}_{k+1}
  \bx_{k+1}\Vert <{\sf tol}$.
\end{enumerate}
\caption{INI for $M$-matrices}
\label{alg:ivaM}
\end{algorithm}

Due to the equivalence of Algorithm~\ref{alg:ivaM} and Algorithm~\ref%
{alg:iiva}, the previous convergence results for the irreducible nonnegative
matrix eigenvalue problem naturally hold for the irreducible nonsingular $M$%
-matrix eigenvalue problem under consideration. We summarize the main
results as follows.

\begin{theorem}
Let $A$ be an irreducible nonsingular $M$-matrix. If $\underline{\lambda}_k$
and $\mathbf{x}_k$ are generated by Algorithm~\ref{alg:ivaM}, then $\{
\underline{\lambda}_k\} \rightarrow \lambda $, the smallest eigenvalue of $A$,
monotonically from below as $k\rightarrow \infty$, and $\underset{%
k\rightarrow \infty}{\lim}\mathbf{x}_k = \mathbf{x}$ with $\mathbf{x}_k>
\mathbf{0}$ for all $k>0$. Furthermore, the convergence of INI\_1 and
INI\_2 is globally linear with the asymptotic convergence factor
bounded by $\frac{2\gamma}{1+\gamma}$ and superlinear with the convergence order
$\frac{1+\sqrt{5}}{2}$, respectively.
\end{theorem}

\section{Numerical experiments}

\label{sec:exp} In this section we present numerical experiments to support
our theoretical results on INI and NI, and illustrate the effectiveness of
the proposed INI algorithms. In the meantime, we compare INI with the
algorithms JDQR \cite{S02}, JDRPCG \cite{Not04}, the implicitly restarted
Arnoldi method \cite{Sor92}, i.e., the Matlab function \textsf{eigs}, and
the explicitly Krylov--Schur method \cite{Ste02}, whose Matlab code is
\textsf{krylovschur} downloaded from \cite{slepc07}, all of which are not
positivity preserving for approximate eigenvectors. We use JDQR for the
nonnegative matrix and $M$-matrix eigenvalue problems, but we use JDRPCG
only for symmetric nonsingular $M$-matrices since JDRPCG is designed to
compute a few number of smallest eigenpairs of a symmetric matrix. We show
that the NI and INI algorithms are always reliable to compute positive
eigenvectors while the other algorithms generally fail to do so. Actually,
for the three ones of four problems tested, the converged eigenvectors
obtained by the other algorithms were not positive. We also demonstrate that
the INI algorithms are efficient, and they are competitive with and can be
considerably efficient than the others. All numerical tests were performed
on an Intel (R) Core (TM) i$5$ CPU $750$@ $2.67$GHz with $4$ GB memory using
Matlab $7.11.0$ with the machine precision $\epsilon=2.22\times 10^{-16}$
under the Microsoft Windows $7$ $64$-bit.

\subsection{INI for Nonnegative Matrices}

We present two examples to illustrate numerical behavior of NI, INI\_1 and
INI\_2 for nonnegative matrices. At each outer iteration the approximate
solution $\mathbf{y}_{k+1}$ of (\ref{eq:inexactsys}) satisfies%
\begin{equation*}
(\overline{\lambda}_kI-B) \, \mathbf{y}_{k+1}= \mathbf{x}_k+\mathbf{f}_k
\end{equation*}
by requiring the following inner tolerances:

\begin{itemize}
\item for NI: $\Vert \mathbf{f}_k\Vert \le 10^{-14}$;

\item for INI\_1: $\Vert \mathbf{f}_k\Vert \le \gamma\min(\mathbf{x}_k)$
with some $0< \gamma< 1$;

\item for INI\_2: $\Vert \mathbf{f}_k\Vert \le \min\{\gamma \min(\mathbf{x}%
_k),\frac{\overline{\lambda}_{k-1}-\overline{\lambda}_k}{\overline{\lambda}%
_{k-1}}\}$ for $k \ge 1$ and $\Vert \mathbf{f}_0\Vert\le \gamma \min(\mathbf{%
x}_0)$ with some $0< \gamma <1$.
\end{itemize}

We explain more on the inner tolerance used in INI\_2. Recall that we always
have $\min(\mathbf{x}_k)\le n^{-1/2}$, which is reasonably small for $n$
very large. Therefore, starting with a general unit length vector $\mathbf{x}%
_0>\mathbf{0}$, the sequence $\{\overline{\lambda}_k\}$ by INI\_2 with $\|%
\mathbf{f}_k\|\leq \frac{\overline{\lambda}_{k-1}-\overline{\lambda}_k}{%
\overline{\lambda}_{k-1}}$ may not satisfy $\vert \mathbf{f}_k\vert\le
\gamma \mathbf{x}_k$, the condition of Lemma~\ref{monotone}. Consequently,
the global convergence of $\{\overline{\lambda}_k\}$ is not guaranteed.
However, we must have $\vert \mathbf{f}_k\vert\le \gamma \mathbf{x}_k$ for
the above-proposed inner tolerance for INI\_2, such that $\{\overline{\lambda%
}_k\}$ generated by INI\_2 is globally convergent. Furthermore, once $%
\overline{\lambda}_k$ has converged with certain accuracy, we will have $%
\frac{\overline{\lambda}_{k-1}- \overline{\lambda}_k}{\overline{\lambda}%
_{k-1}}<\gamma \min(\mathbf{x}_k)$ for $k$ large enough. After it, it is
known from Theorem~\ref{quadconvini} that INI\_2 will asymptotically
converge superlinearly.

In implementations, it is necessary to impose some guards on inner
tolerances so as to avoid them being too small in finite precision
arithmetic, i.e., below the level of $\epsilon$. For all examples, the
stopping criteria for inner iterations are always taken as%
\begin{equation*}
\Vert \mathbf{f}_k\Vert \le \max \{\gamma\min (x_{k}),10^{-13}\}\text{ for
INI\_1}
\end{equation*}%
and%
\begin{equation*}
\Vert \mathbf{f}_k\Vert \le \max \{\min \{\gamma \min (\mathbf{x}_{k}),
\frac{\overline{\lambda }_{k-1}-\overline{\lambda }_{k}} {\overline{\lambda }%
_{k-1}}\},10^{-13}\}\text{ for INI\_2.}
\end{equation*}
For each example, we test two values of $\gamma=0.8,\ 0.1$ to observe the
effect of $\gamma$ on the convergence of INI\_1.

In the experiments, the stopping criterion for outer iterations is
\begin{equation*}
\frac{\Vert B\mathbf{x}_k-\overline{\lambda}_k\mathbf{x}_k\Vert} {%
(\|B\|_1\|B\|_{\infty})^{1/2}} \le 10^{-13},
\end{equation*}
where we use the cheaply computable $(\|B\|_1\|B\|_{\infty})^{1/2}$ to
estimate the 2-norm $\|B\|$, which is more reasonable than the individual $%
\|B\|_1$ or $\|B\|_{\infty}$ with $\|\cdot\|_{\infty}$ the infinity norm of
a matrix.

For INI and NI, since the coefficient matrices $\overline{\lambda}_k I-B$
are always positive definite for all $k$ when $B$ is symmetric, we use the
conjugate gradient method to solve inner linear systems. For $B$
unsymmetric, we use BICGSTAB as inner solver. In implementations, we use the
standard Matlab functions \textsf{bicgstab} and \textsf{pcg}. The outer
iteration starts with the normalized vector $\frac{1}{\sqrt{n}}\left[%
1,\dots,1\right]^T$ for NI, INI, \textsf{eigs}, and \textsf{krylovschur}.

The original codes of JDQR and JDRPCG use the absolute residual norms to
decide the convergence, which is not robust for a general purpose. By
setting the stopping criteria ``TOL$=10^{-13}{(\|B\|_1\|B\|_{\infty})^{1/2}}$%
'' for outer iterations in them, we will get the same stopping criteria as
that for NI and INI. We set the parameters ``sigma=LM'' and the inner solver
``OPTIONS.LSolver=bicgstab'' in the unsymmetric case and
``OPTIONS.LSolver=minres'' in the symmetric case, where the Matlab function
\textsf{minres} is the minimal residual method. All the other options use
defaults. We do not use any preconditioning for inner linear systems. For
\textsf{eigs} and \textsf{krylovschur}, we set the stopping criteria
``OPTS.tol$=10^{-13}/{(\|B\|_1\|B\|_{\infty})^{1/2}}$, and take the maximum
and minimum subspace dimensions as $20$ and $2$ at each restart,
respectively. For the computation of Perron roots and vectors, we mention
that \textsf{eigs} and \textsf{krylovschur} do not involve shift-invert, so
we do not solve any inner linear systems.

We denote by $I_{\mathrm{outer}}$ the number of outer iterations to achieve
the convergence and by $I_{\mathrm{inner}}$ the total number of inner
iterations. Note that each outer iteration of these algorithms needs one
matrix-vector product formed with $B$, while each iteration of BICGSTAB uses
two matrix-vector products formed with $B$ and $B^T$, respectively, and each
iteration of MINRES or CG uses one matrix-vector product with $B$. We denote
by $I_{total}$ the total matrix-vector products formed with $B$ and possibly
$B^T$, which can fairly measure the overall efficiency of NI, INI, JDQR,
\textsf{eigs} and \textsf{krylovschur} when the subspace dimensions used by
JDQR, \textsf{eigs} or \textsf{krylovschur} are small at each restart. In
view of the above, we have $I_{total}=I_{outer}+2I_{inner}$ for $B$
unsymmetric and $I_{total}=I_{outer}+I_{inner}$ for $B$ symmetric for our
test algorithms. For \textsf{eigs} and \textsf{krylovschur} we have $%
I_{total}=I_{outer}$ since no inner linear system is solved.

Recall that for unsymmetric problems JDQR, \textsf{eigs} and \textsf{%
krylovschur} may compute complex approximate Perron vectors. If it is the
case, we are only concerned with the signs of components in its real part.
To investigate the positivity of a converged Perron vector, we let the real
part of its maximal component in magnitude be positive, no matter how real
or complex it is. In the tables, \textquotedblleft
Positivity\textquotedblright  records whether the converged Perron vector
preserves the strict positivity property. If no, then the percentage in the
brace indicates the proportion that the converged Perron vector has the
components with positive real part. We also report the CPU time of each
algorithm, which measures the overall efficiency too.


\begin{example}
\label{exp:google}From DIMACS10 test set \cite{DIMACS}, we consider the
unsymmetric nonnegative matrix \textsf{web-Google}. The matrix data
was released in 2002 by Google as a part of Google Programming Contest.
The \textsf{web-Google} is a directed graph with nodes representing web pages and
directed edges representing hyperlinks between them. This matrix is a binary
matrix of order $n=916,428$ and has $5,105,039$ nonzero entries.
\end{example}

Table~\ref{table2} reports the results obtained by NI, INI\_1 with $\gamma
=0.8$ and $\gamma=0.1$, INI\_2, JDQR, \textsf{eigs}, and \textsf{krylovschur}%
. Figures~\ref{fig:google3}--\ref{fig:google2} depict how the residual norms
of outer iterations evolve versus the sum of inner iterations and versus the
outer iterations for NI, INI\_1 with two $\gamma$, INI\_2, and JDQR,
respectively. As we see, Figure~\ref{fig:google2} shows that the residual
norms computed by the five algorithms decreased monotonically. It also
indicates that the NI and INI algorithms converged slowly and similarly in
the beginning of outer iterations. Then they started converging fast. We
find that INI\_1 and INI\_2 achieved the same superlinear (quadratic)
convergence and exhibited very similar convergence behavior to NI, and all
of them used nine outer iterations to achieve the convergence. These results
show that our theory on both INI\_1 and INI\_2 can be conservative.

We observe from Table~\ref{table2} that the INI\_1 and INI\_2 improved the
overall efficiency of NI considerably. Actually, the CPU times used by
INI\_1 with $\gamma=0.8$ and $\gamma=0.1$ were $43\%$ and $51\%$ of those
used by NI, respectively, and that used by INI\_2 was $43\%$ of that used by
NI. In terms of either $I_{total}$ or the CPU time, INI\_1 with $\gamma=0.8$%
, INI\_2 and JDQR were twice as fast as NI. In addition, we find that INI\_1
and INI\_2 were competitive with \textsf{krylovschur}, and they used very
comparable CPU time.

Table~\ref{table2} also shows that \textsf{eigs} used the least CPU time but
more outer iterations than \textsf{krylovschur}, and JDQR was as efficient
as INI\_1 and INI\_2 in terms of $I_{total}$ and the CPU time. The three
algorithms JDQR, \textsf{eigs} and \textsf{krylovschur} computed the Perron
root reliably, but the converged Perron vectors were not positive. As Table~%
\ref{table2} indicates, only 41\%, 42\% and 21\% of the components of the
converged eigenvectors by JDQR, \textsf{krylovschur} and \textsf{eigs} were
positive or had positive real part. This implies that the too many
components were not reliable and were hard to interpret.

We comment that the results by JDQR, \textsf{eigs} and \textsf{krylovschur}
for a desired positive eigenvector are not unusual. Because some
component(s) of the Perron vector must be very small, it is quite possible
that the JDQR, \textsf{eigs} and \textsf{krylovschur} cannot guarantee the
strict positivity of approximate eigenvectors, and, particularly, those very
small or tiny true positive components may change signs and become negative
in the approximations, even though the approximations have already converged
to the positive eigenvector $\mathbf{x}$ in the conventional sense and
attains its maximum accuracy, namely, the level of $\epsilon$.

\begin{table}[tbp]
\caption{The total outer and inner iterations in Example \protect\ref%
{exp:google}}
\label{table2}
\begin{tabular}{l|rrrrl}
\hline
Method & $I_{\mathrm{outer}}$ & $I_{\mathrm{inner}}$ & $I_{total}$ & CPU time
& Positivity \\ \hline
NI & 9 & 115.5 & 240 & 33.4 & Yes \\
INI\_1 with $\gamma =0.8$ & 9 & 55 & 119 & 14.2 & Yes \\
INI\_1 with $\gamma =0.1$ & 9 & 61.5 & 132 & 17.1 & Yes \\
INI\_2 & 9 & 56 & 121 & 14.4 & Yes \\
JDQR & 4 & 52.5 & 109 & 14.2 & No (41\%) \\
\textsf{krylovschur} & 60 & ----- & 60 & 14.5 & No (42\%) \\
\textsf{eigs} & 80 & ----- & 80 & 8.3 & No (21\%) \\ \hline
\end{tabular}%
%
\end{table}

\begin{center}
\begin{figure}[!ht]
\begin{center}
\begin{minipage}[t]{0.49\textwidth}
\begin{center}
\includegraphics[width=6cm]{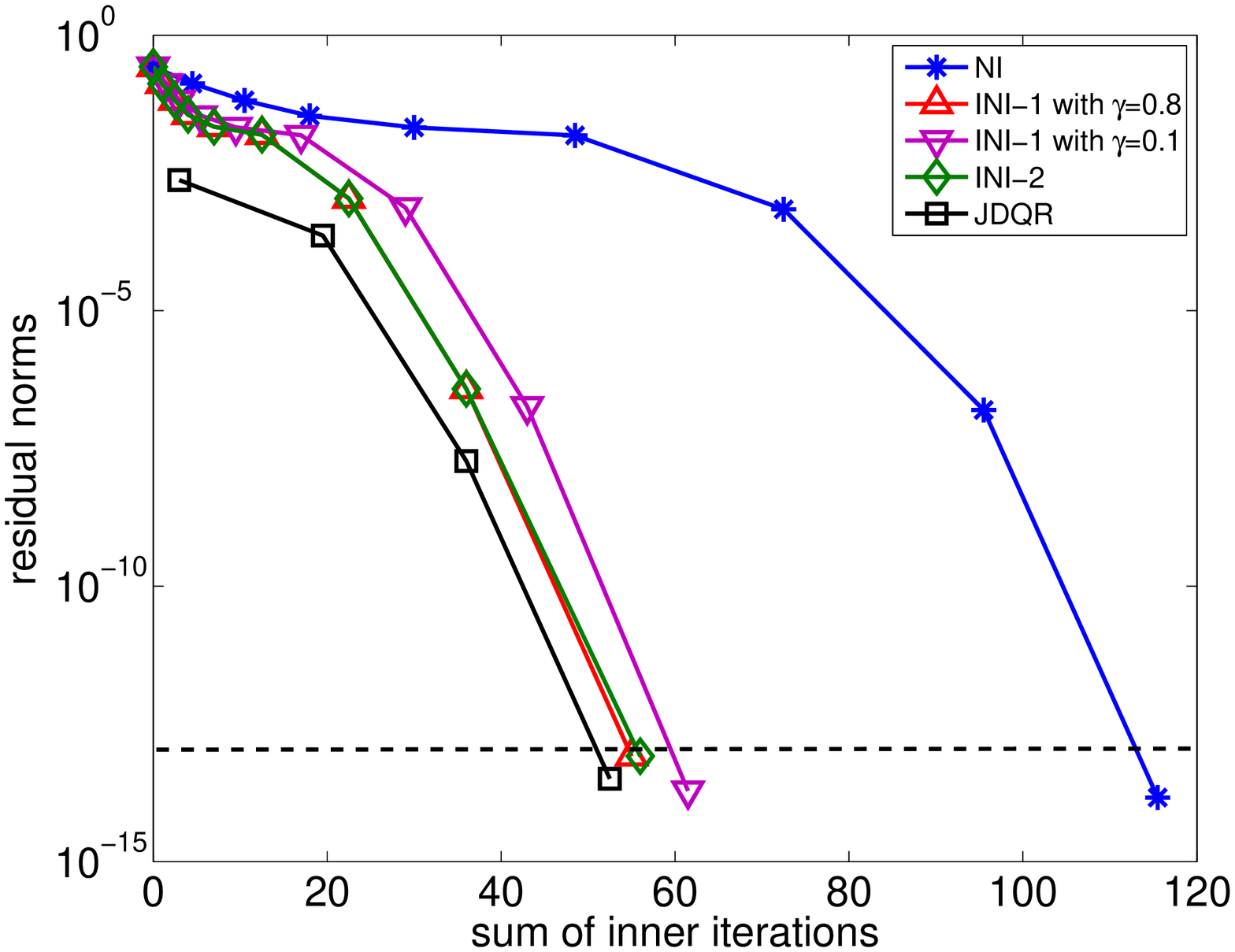}
\caption{\small Example~\ref{exp:google}.
The outer residual norms versus sum of inner iterations.}
\label{fig:google3}
\end{center}
\end{minipage}
\hfill
\begin{minipage}[t]{0.49\textwidth}
\begin{center}
\includegraphics[width=6cm]{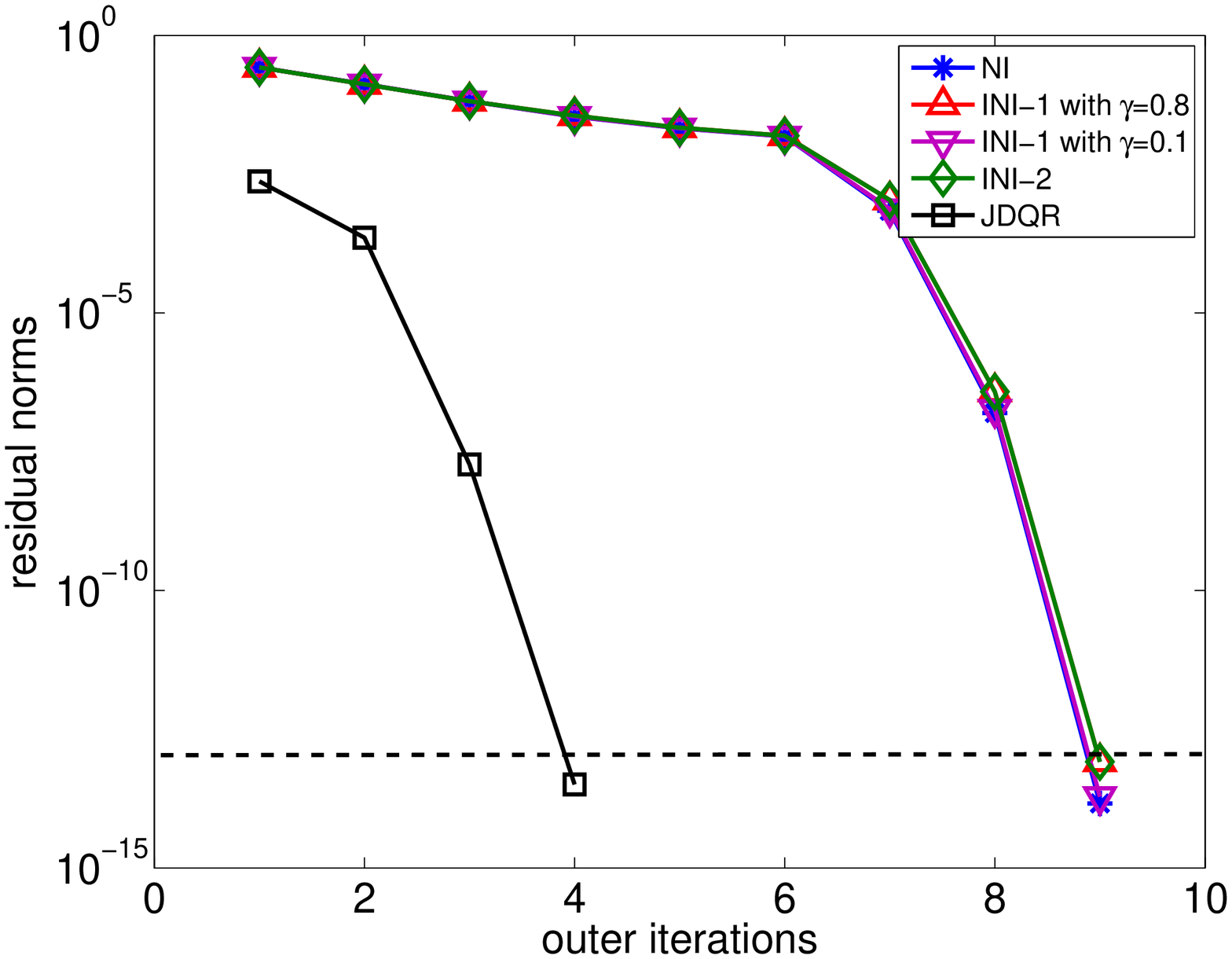}
\caption{\small Example~\ref{exp:google}.
The outer residual norms versus the outer iterations.}
\label{fig:google2}
\end{center}
\end{minipage}
\end{center}
\end{figure}
\end{center}

Finally, it is worth pointing out that, as we have observed, NI needed more
inner iterations than INI\_1 and INI\_2 at each outer iteration, and the
inner iterations used by each algorithm increased as outer iterations
proceeded. This is due to the fact that inner linear system (\ref{eq:step1})
was worse conditioned with increasing $k$ as$\overline{\lambda}_k$ was
closer to $\lambda$, causing that more inner iterations are generally needed
for the fixed $\gamma$, i.e., INI\_1. The situation is more serious with
higher accuracy, as required by INI\_2 and JDQR. Even so, Figure~\ref%
{exp:google} illustrates that INI and JDQR converged superlinearly as inner
iterations increased.


\begin{example}
\label{delaunay}Consider the symmetric nonnegative matrix
\textsf{delaunay\_n20} from DIMACS10 test set \cite{DIMACS}. The
matrix is generated by Delaunay triangulations of random points in the unit
square. It is a binary matrix of order $n=2^{20}=1,048,576$
with $6,291,372$ nonzero entries.
\end{example}

Table~\ref{table3} and Figures~\ref{fig:del3}--\ref{fig:del2} report the
results and convergence processes.

First of all, Table~\ref{table3} shows that JDQR, \textsf{eigs} and \textsf{%
krylovschur} were unreliable to compute the Perron vector and were unable to
produce the positive converged Perron vectors, though \textsf{eigs} and
\textsf{krylovschur} were equally the most efficient and were about $1.2\sim
1.5$ times as fast as INI in terms of $I_{total}$ and the CPU time. It is
seen that only 82\%, 55\% and 59\% of the components of the finally
converged approximations by JDQR, \textsf{krylovschur} and \textsf{eigs}
were positive or had positive real part, respectively. JDQR was the slowest,
and NI was the second most expensive in terms of both $I_{total}$ and the
CPU time. INI\_1 with two $\gamma$ and INI\_2 cost only about $50\%$ of NI,
33\% of JDQR, a considerable improvement over NI and JDQR.

Neglecting the positivity preservation of the converged Perron vectors, we
see that Table~\ref{table3} and Figure~\ref{fig:del2} indicate that NI,
INI\_2, INI\_1 with two $\gamma$ and JDQR used exactly the same, i.e., ten
outer iterations. As is seen from Figure~\ref{fig:del2}, all the algorithms
converged slowly until the sixth outer iteration, then they speeded up very
considerably and converged superlinearly. Furthermore, from the sixth outer
iteration onwards, Figure~\ref{fig:del2} clearly demonstrates that each
algorithm actually achieved the quadratic convergence. It is worthwhile to
point out that, unlike NI and INI, which converged monotonically, JDQR
exhibited irregular convergence behavior. Precisely, the residual norms of
JDQR decreased at the first four outer iterations, increased at the two
outer iterations followed, and then converged regularly from the sixth outer
iteration upwards. Besides, Figure~\ref{fig:del3} indicates that NI, INI and
JDQR converged slowly and linearly in the initial stage, and then they
converged increasingly faster, i.e., superlinearly as inner iterations
increased.


\begin{table}[tbp]
\caption{The total outer and inner iterations in Example \protect\ref%
{delaunay}}
\label{table3}
\begin{tabular}{l|rrrrl}
\hline
Method & $I_{\mathrm{outer}}$ & $I_{\mathrm{inner}}$ & $I_{total}$ & CPU time
& Positivity \\ \hline
NI & 10 & 524 & 534 & 86 & Yes \\
INI\_1 with $\gamma =0.8$ & 10 & 259 & 269 & 40 & Yes \\
INI\_1 with $\gamma =0.1$ & 10 & 291 & 301 & 46 & Yes \\
INI\_2 & 10 & 261 & 271 & 40 & Yes \\
JDQR with Minres & 9 & 750 & 759 & 121 & No (82\%) \\
\textsf{krylovschur} & 180 & ----- & 180 & 28 & No (55\%) \\
\textsf{eigs} & 200 & ----- & 200 & 24 & No (59\%) \\ \hline
\end{tabular}%
%
\end{table}

\begin{center}
\begin{figure}[!ht]
\begin{center}
\begin{minipage}[t]{0.49\textwidth}
\begin{center}
\includegraphics[width=6cm]{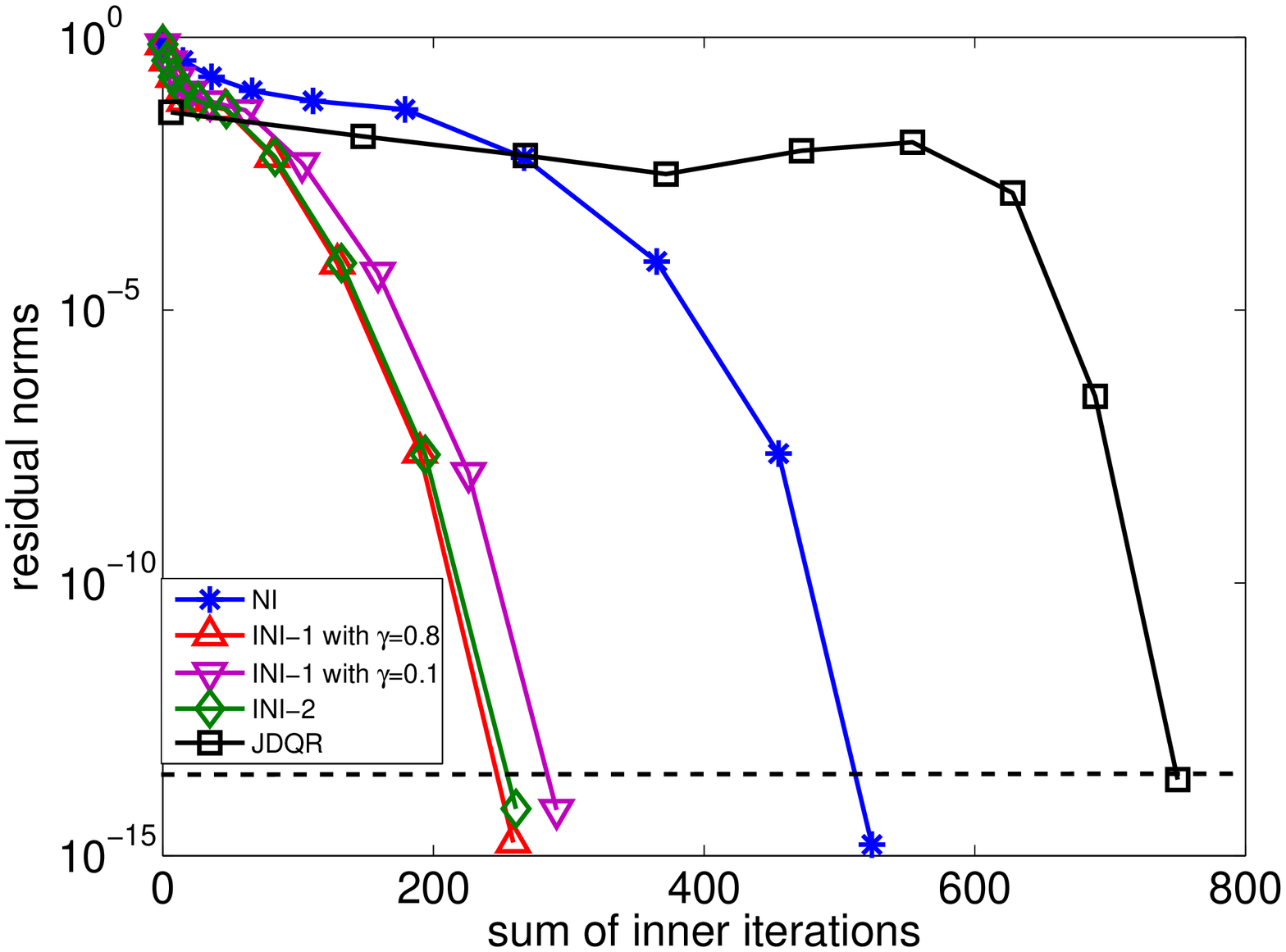}
\caption{\small Example~\ref{delaunay}.
The outer residual norms versus sum of inner iterations.}
\label{fig:del3}
\end{center}
\end{minipage}
\hfill
\begin{minipage}[t]{0.49\textwidth}
\begin{center}
\includegraphics[width=6cm]{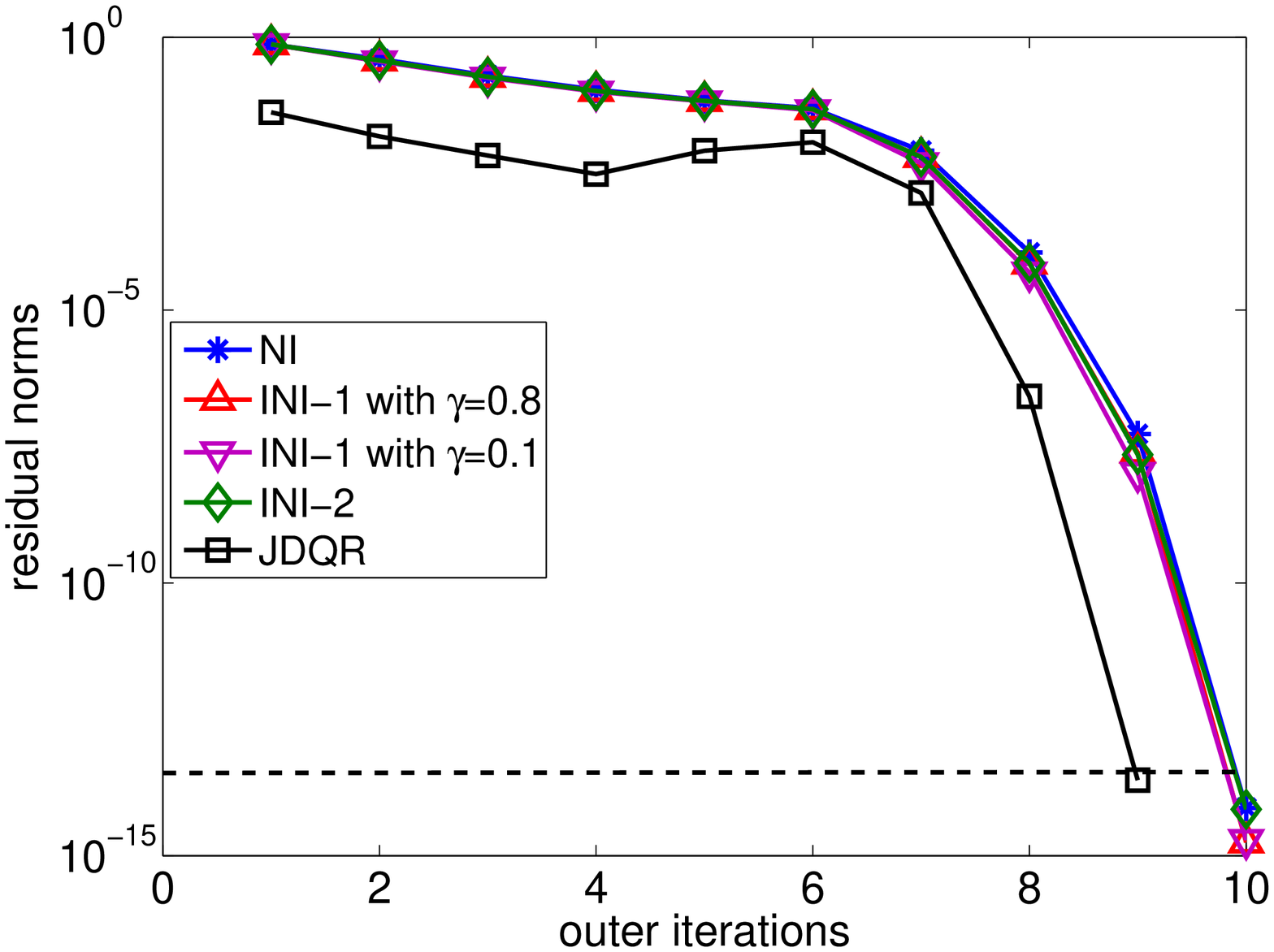}
\caption{\small  Example~\ref{delaunay}.
The outer residual norms versus the outer iterations.}
\label{fig:del2}
\end{center}
\end{minipage}
\end{center}
\end{figure}
\end{center}

\subsection{INI for M-matrices}

In this subsection, we use NI and INI to compute the smallest eigenpair of
an irreducible nonsingular $M$-matrix $A$ and illustrate the effectiveness
of INI. For NI, INI\_1 and INI\_2, the stopping criteria for inner and outer
iterations are the same as those for the nonnegative matrix eigenproblem.
For JDQR, we also take the same stopping criterion for outer iterations, and
set the parameter ``sigma=SM''. All the other parameters are the same as
those for nonnegative matrix eigenvalue problems. For \textsf{eigs}, since
we now compute the smallest eigenvalue and the associated eigenvector of an
irreducible nonsingular $M$-matrix $A$, we use \textsf{eigs}($\sigma
speye(n)-A$,'LM',OPTS) to replace with \textsf{eigs}($A$,'SM',OPTS) for some
easily chosen $\sigma>\rho(A)$. The reason for this strategy is twofold:
First, $\sigma I-A$ becomes an irreducible nonnegative matrix, which means
that we only need to compute the largest eigenpair of $\sigma I-A$ without
solving any linear system. Second, if \textsf{eigs}($A$,'SM',OPTS) is used,
we have to solve a sequence of inner linear systems by first making a sparse
LU factorization of $A$ and then solving two lower and upper triangular
systems at each iteration. This is assumed to be prohibited for a very large
$A$ throughout this paper. We also apply \textsf{krylovschur} to $\sigma I-A$
to find the smallest eigenpair of $A$.


\begin{example}
\label{exp:3Dmesh}We consider an unsymmetric irreducible nonsingular
$M$-matrix from the 3D Human
Face Mesh \cite{Yau13} with a small noise. This is a matrix of order
$n=42,875$ with $ 693,875$ nonzero entries.
\end{example}

We choose $\sigma$ in such a way: let $d=\max_{1\leq i\leq n}\{a_{ii}\}$,
which equals 2771, and take $\sigma=3000$. We then get a nonnegative matrix $%
B=\sigma I-A$. Table~\ref{table5} reports the results obtained, and Figures~%
\ref{fig:mesh}--\ref{fig:mesh2} depict the convergence processes of NI, INI
and JDQR.

For this problem, the good news is that all the algorithms worked well and
all the converged eigenvectors were positive. We observe from the figures
that NI, INI and JDQR converged smoothly and fast as inner iterations
increased or outer iterations proceeded. Except INI\_1 with $\gamma=0.8$,
which used ten outer iterations, NI and INI used seven outer iterations to
attain the desired accuracy. However, regarding the overall efficiency, NI
was the most expensive, and INI\_2 was considerably better than NI and used $%
65\%$ of $I_{total}$ and $67\%$ of the CPU time that NI consumed. INI\_1
with two $\gamma$ improved the performance of NI more substantially. INI\_1
with $\gamma =0.8$ was the most efficient in terms of $I_{total}$ and the
CPU time. \textsf{eigs} and \textsf{krylovschur} were similarly efficient
and consumed almost the same CPU time as INI\_1 with $\gamma=0.8$, but they
used considerably more $I_{total}$. In addition, we have observed that
INI\_1 with $\gamma=0.1$ and INI\_2 were competitive with JDQR.

Finally, we see that for this $M$-matrix two different $\gamma$ led to
distinct convergence behavior. As Figure~\ref{fig:mesh2} indicates, NI and
INI\_2 typically converged superlinearly, and INI\_1 with $\gamma=0.1$
exhibited very similar superlinear convergence to NI, while INI\_1 with $%
\gamma=0.8$ typically converged linearly. These confirm our theory and
demonstrate that our theoretical results can be realistic and pronounced.

\begin{table}[tbp]
\caption{The total outer and inner iterations in Example \protect\ref%
{exp:3Dmesh}}
\label{table5}
\begin{tabular}{l|rrrrl}
\hline
Method & $I_{\mathrm{outer}}$ & $I_{\mathrm{inner}}$ & $I_{total}$ & CPU time
& Positivity \\ \hline
NI & 7 & 1400 & 2807 & 19.1 & Yes \\
INI\_1 with $\gamma =0.8$ & 10 & 521 & 1052 & 7.2 & Yes \\
INI\_1 with $\gamma =0.1$ & 7 & 773 & 1553 & 10.2 & Yes \\
INI\_2 & 7 & 912 & 1831 & 12.9 & Yes \\
JDQR & 6 & 506 & 1018 & 10.9 & Yes \\
\textsf{krylovschur} & 1160 & ----- & 1160 & 7.4 & Yes \\
\textsf{eigs} & 1420 & ----- & 1420 & 7.0 & Yes \\ \hline
\end{tabular}%
%
\end{table}

\begin{center}
\begin{figure}[!ht]
\begin{center}
\begin{minipage}[t]{0.49\textwidth}
\begin{center}
\includegraphics[width=6cm]{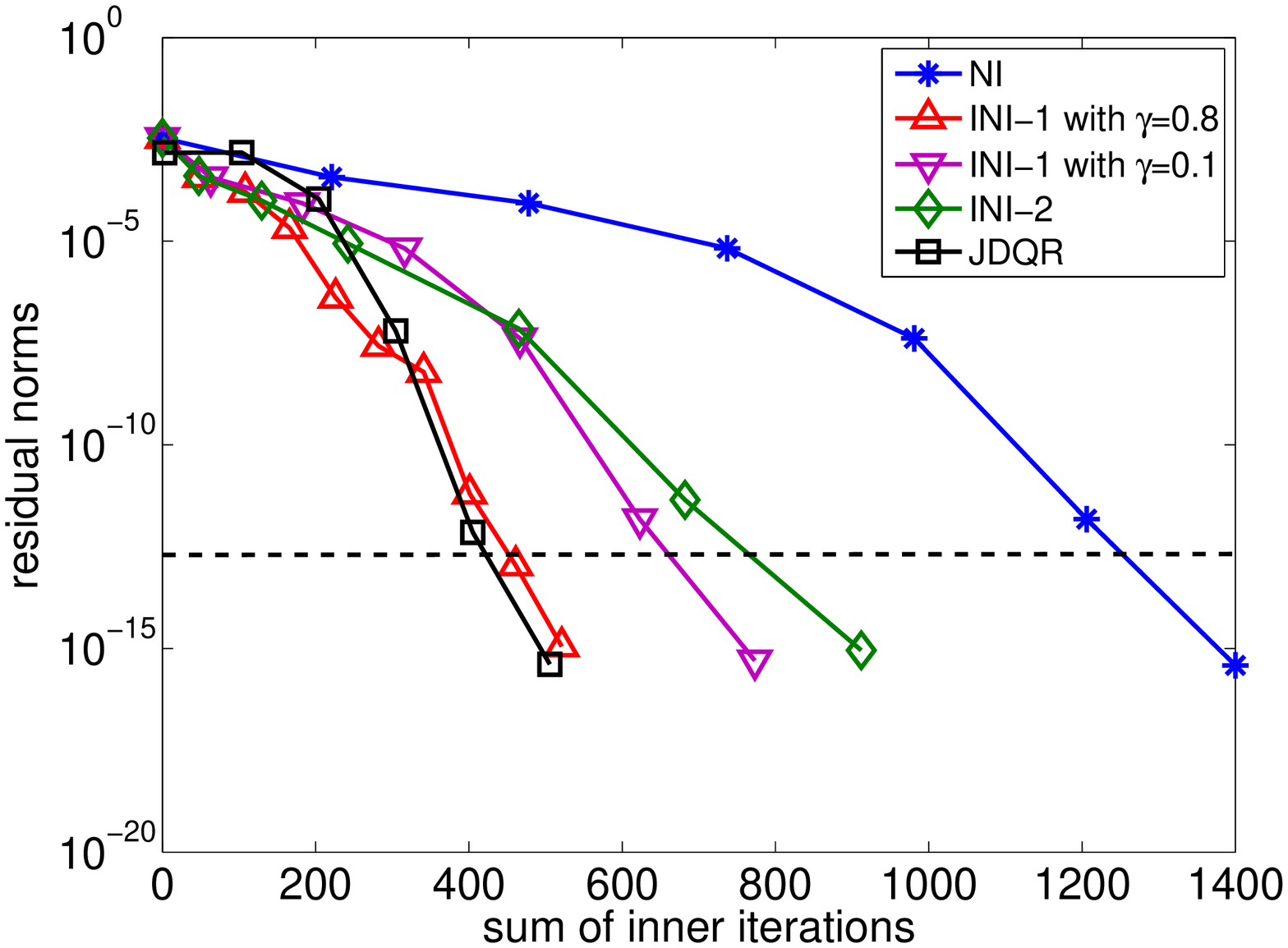}
\caption{\small Example~\ref{exp:3Dmesh}.
The outer residual norms versus sum of inner iterations.}
\label{fig:mesh}
\end{center}
\end{minipage}
\hfill
\begin{minipage}[t]{0.49\textwidth}
\begin{center}
\includegraphics[width=6cm]{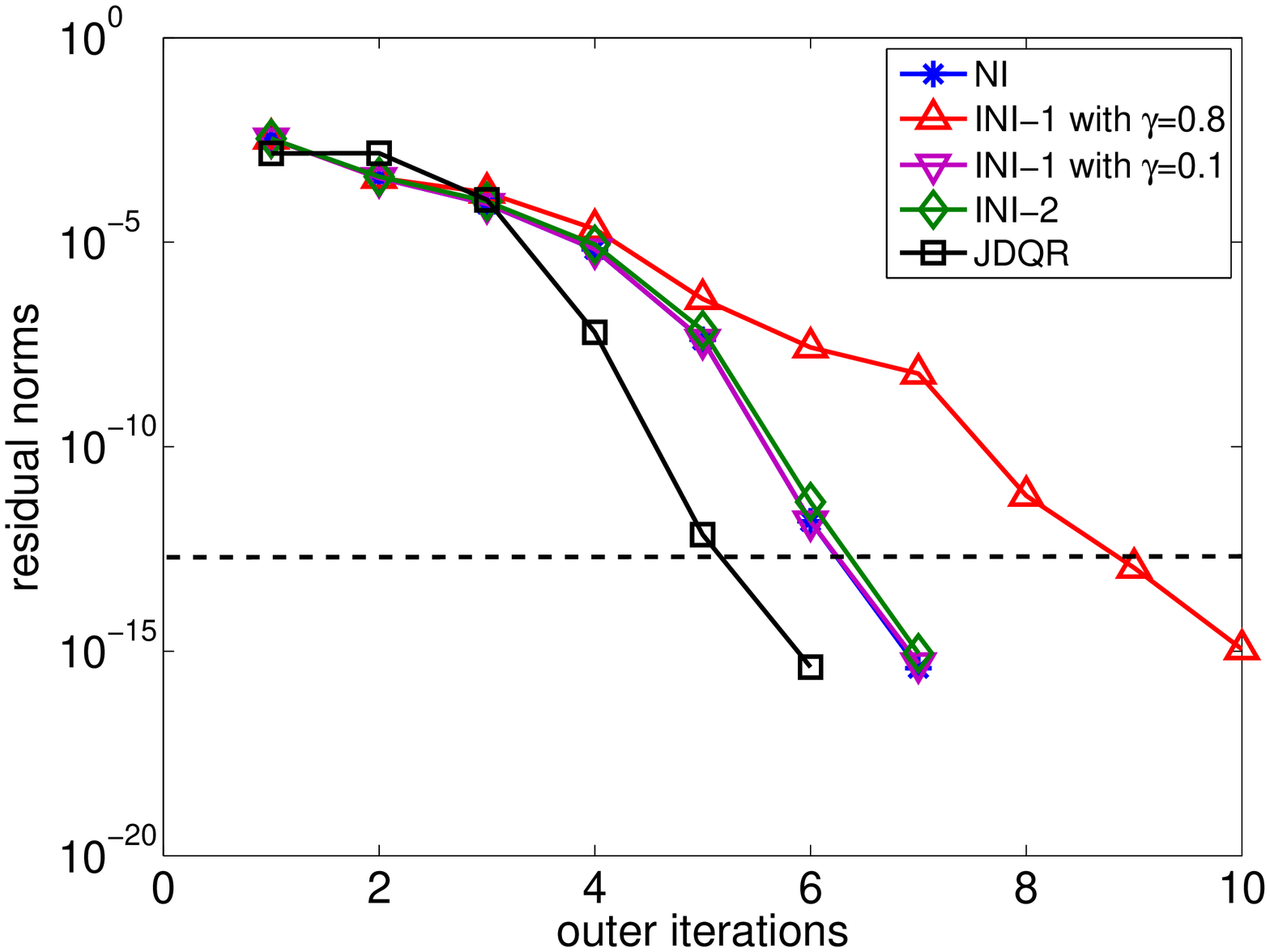}
\caption{\small Example~\ref{exp:3Dmesh}.
The outer residual norms versus the outer iterations.}
\label{fig:mesh2}
\end{center}
\end{minipage}
\end{center}
\end{figure}
\end{center}

\begin{example}
\label{exp:nicolo}Consider the symmetric $M$-matrix \textsf{nicolo\_da\_uzzano}
from AIM@SHAPE Shape Repository \cite{aim}. This matrix describes the full
resolution shape (2 millions triangles) edited to remove the errors introduced
by scanning and reconstruction phases. The matrix \textsf{nicolo\_da\_uzzano} is
generated by the barycentric mapping method \cite{wei10}, is of order
$n=943,870$ and has $8,960,880$ nonzero entries.
\end{example}

For \textsf{eigs} and \textsf{krylovschur}, we get a nonnegative matrix $%
B=\sigma I-A$ in this way: let $d=\max_{1\leq i \leq n} \{a_{ii}\}$, which
equals 2421, and then take $\sigma=2500$. Table~\ref{table6} reports the
numerical results obtained by NI, INI, JDQR, JDRPCG, \textsf{krylovschur}
and \textsf{eigs}. Figures~\ref{fig:nico3}--\ref{fig:nico2} describe the
convergence processes of NI and INI. We omit the convergence curves of JDQR
and JDRPCG because JDQR failed to compute the desired eigenpair and JDRPCG
used much more outer iterations than NI and INI did, as is seen from Table~%
\ref{table6}.

As far as $I_{outer}$ is concerned, NI and INI worked very well and used
only four outer iterations to attain the desired accuracy. As Figure~\ref%
{fig:nico2} shows, they had indistinguishable convergence curves.
Furthermore, all of them were reliable and positivity preserving. For the
overall performance, INI\_1 and INI\_2 were equally efficient, and they used
about 60\% of $I_{total}$ and the CPU time of NI.

In contrast, JDQR, JDRPCG, \textsf{krylovschur} and \textsf{eigs} performed
poorly for this problem. At much more expenses than NI and INI, JDQR finally
failed to produce the correct eigenpair and misconverged. JDRPCG, \textsf{%
kryllovschur} and \textsf{eigs} computed the desired eigenvalue, but the
converged eigenvectors were not positive. Table~\ref{table6} indicates that
for these algorithms roughly 50\% of the components of each converged
eigenvector were negative.

Without considering the reliability or the positivity, for the overall
efficiency, JDRPCG was competitive with INI in terms of the CPU time and $%
I_{total}$. However, \textsf{krylovschur} and \textsf{eigs} used as twice
the CPU time as INI; they were also more expensive than INI in terms of $%
I_{total}$. \textsf{krylovschur} used considerably fewer $I_{total}$ than
\textsf{eigs} did, but it consumed more CPU time than the latter.

\begin{table}[tbp]
\caption{The total outer and inner iterations in Example \protect\ref%
{exp:nicolo}}
\label{table6}
\begin{tabular}{l|rrrrl}
\hline
Method & $I_{\mathrm{outer}}$ & $I_{\mathrm{inner}}$ & $I_{total}$ & CPU time
& Positivity \\ \hline
NI & 4 & 697 & 701 & 54 & Yes \\
INI\_1 with $\gamma =0.8$ & 4 & 400 & 404 & 31 & Yes \\
INI\_1 with $\gamma =0.1$ & 4 & 440 & 444 & 33 & Yes \\
INI\_2 & 4 & 405 & 409 & 31 & Yes \\
JDQR with Minres & 53 & 10458 & 10511 & 1480 & Wrong eigenpair \\
JDRPCG & 146 & 279 & 425 & 23 & No (57\%) \\
\textsf{krylovschur} & 500 & ----- & 500 & 75 & No (45\%) \\
\textsf{eigs} & 820 & ----- & 820 & 60 & No (42\%) \\ \hline
\end{tabular}%
%
\end{table}

\begin{center}
\begin{figure}[!ht]
\begin{center}
\begin{minipage}[t]{0.49\textwidth}
\begin{center}
\includegraphics[width=6cm]{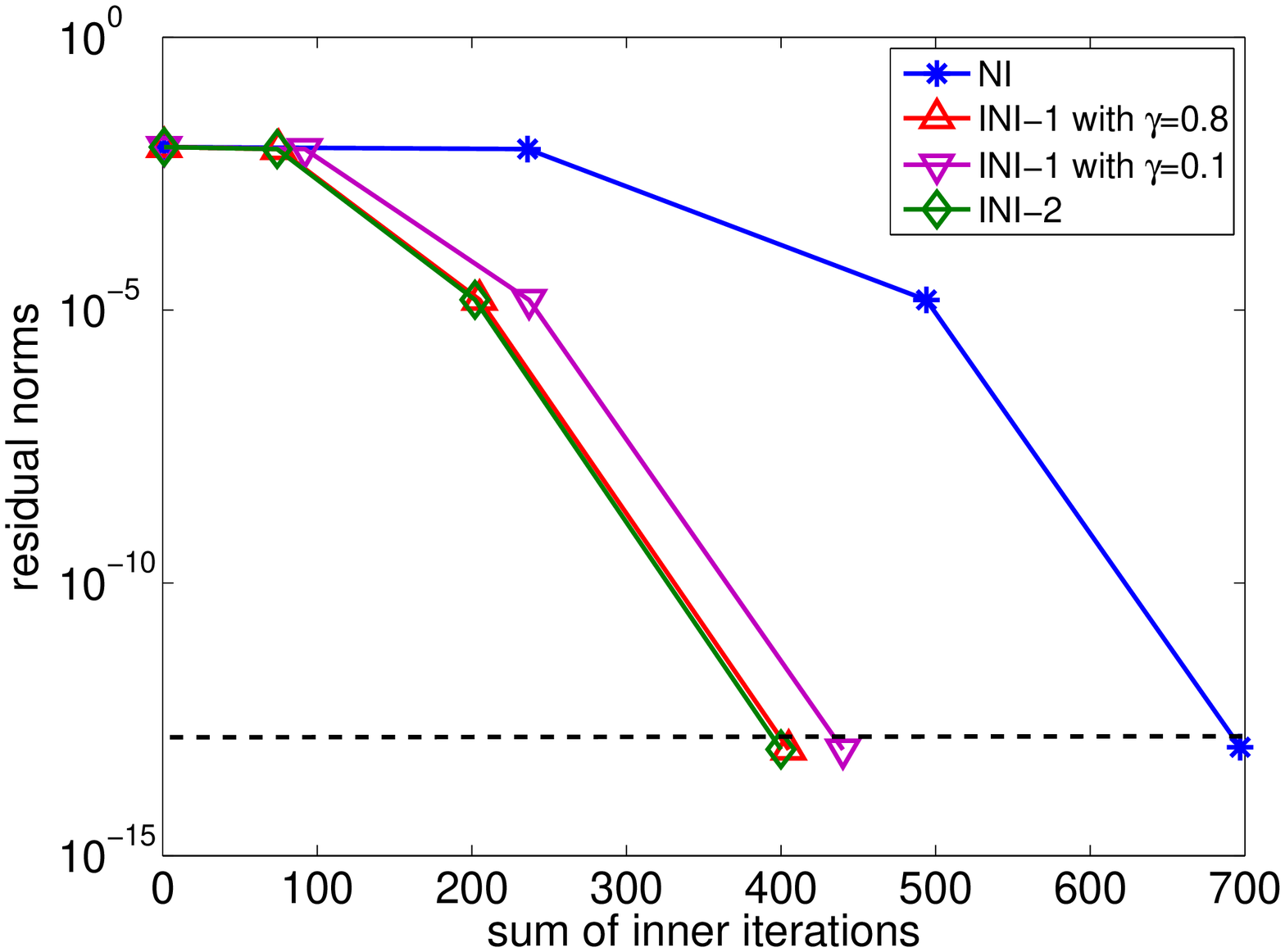}
\caption{\small Example~\ref{exp:nicolo}.
The outer residual norms versus sum of inner iterations.}
\label{fig:nico3}
\end{center}
\end{minipage}
\hfill
\begin{minipage}[t]{0.49\textwidth}
\begin{center}
\includegraphics[width=6cm]{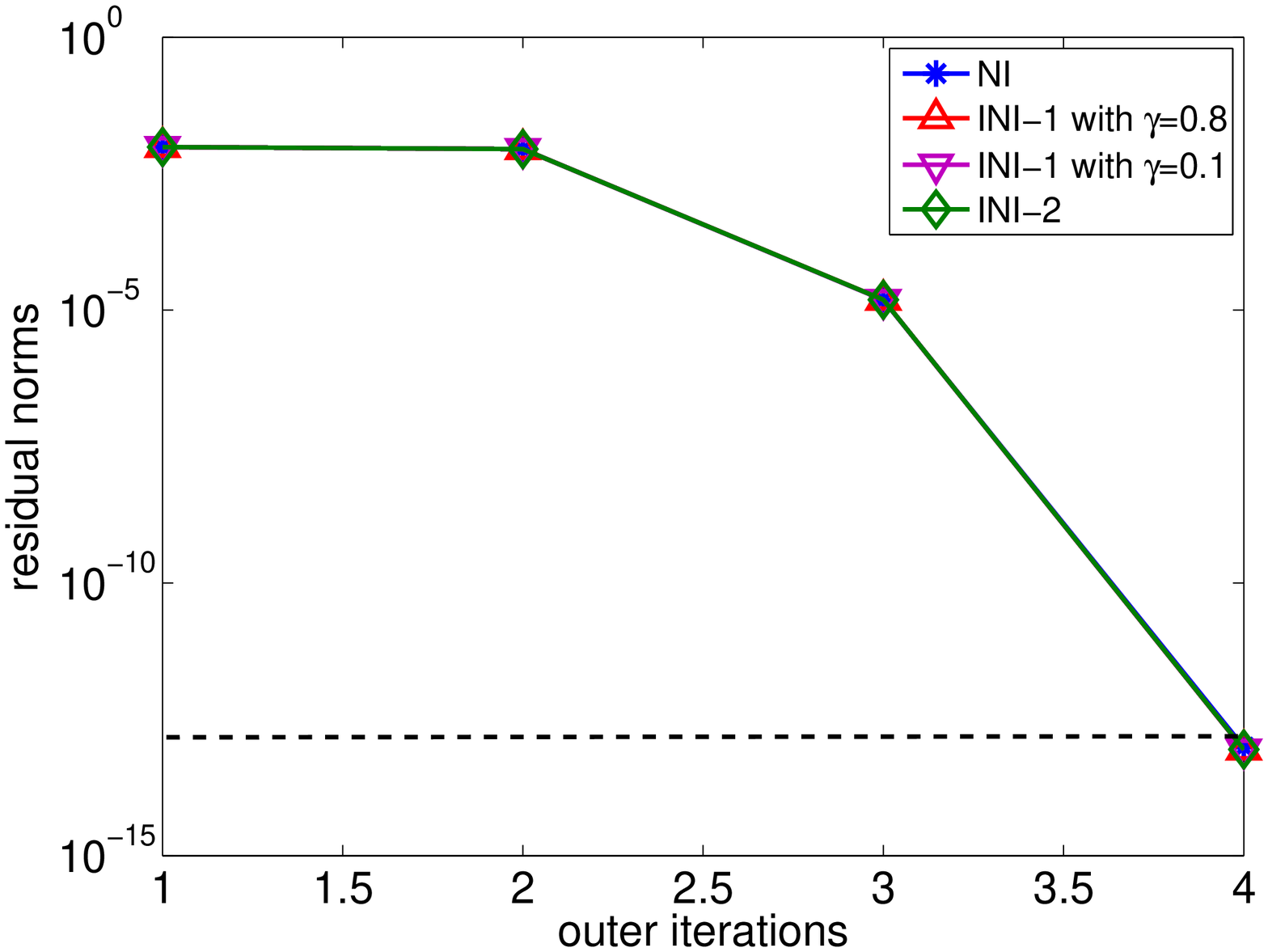}
\caption{\small Example~\ref{exp:nicolo}.
The outer residual norms versus the outer iterations.}
\label{fig:nico2}
\end{center}
\end{minipage}
\end{center}
\end{figure}
\end{center}

\begin{remark}
We have also tested the previous four examples using the power method,
which preserves the positivity of the approximate eigenvectors at each iteration
and is linearly convergent for any positive starting vector.
For Examples~\ref{exp:google}--\ref{delaunay},
we have found that the power method needed $3\sim 6$ times of the CPU time of
INI\_1 because of the eigenvalue clustering. For $M$-matrices in
Examples \ref{exp:3Dmesh} and \ref{exp:nicolo}, we applied the power method to
the nonnegative matrices $B$, which were generated by $B=\sigma I-A$ with
$\sigma$ chosen as before, for finding $\rho(B)$ and the Perron vectors.
The power method cost more than ten times of the CPU time of INI\_1.
So INI was much more efficient than the power method for these four
examples. We omit the details on numerical results obtained by the power method.
\end{remark}

\section{Conclusions}

For the efficient computation of the smallest eigenpair of a large
irreducible nonsingular $M$-matrix, we have proposed a positivity preserving
inexact Noda iteration method with two practical inner tolerance strategies
provided for solving the linear systems involved. We have analyzed the
convergence of the method in detail, and have established a number of global
linear and superlinear convergence results, with the linear convergence
factor and the superlinear convergence order derived explicitly. Precisely,
we have proved that INI converges at least linearly with the asymptotic
convergence factor bounded by $\frac{2\gamma}{1+\gamma}$ for $\xi_k = \Vert
\mathbf{f}_k\Vert\le \gamma \min(\mathbf{x}_k)$ and superlinearly with the
convergence order $\frac{1+\sqrt{5}}{2}$ for the decreasing $\Vert \mathbf{f}%
_k\Vert \leq \frac{\underline{\lambda}_k- \underline{\lambda}_{k-1}}{%
\underline{\lambda}_k}=O(\tan\varphi_{k-1})$, respectively. We have also
revisited the convergence of NI and proved its quadratic convergence in a
different form from \cite{Els76}. The results on INI clearly show how inner
tolerance affects the convergence of outer iterations.

Numerically, we have illustrated that the proposed INI algorithms are
practical for large nonnegative matrix and $M$-matrix eigenvalue problems,
and they can reduce the total computational cost of NI substantially. In the
meantime, regarding the positivity preservation, the numerical experiments
have shown that the INI algorithms are superior to the Jacobi--Davidson
method, the implicitly restarted Arnoldi method and the Krylov--Schur
method. They always preserve the positivity of approximate eigenvectors,
while the Jacobi--Davidson method, the implicitly restarted Arnoldi method
and the Krylov--Schur method often fail to do so. Moreover, regarding he
overall efficiency, INI algorithms are competitive with and can be
considerably higher than the other three methods for some practical problems.

\begin{acknowledgements}
We thank the referees for their careful reading of our paper and
a number of valuable comments, which helped us improve the presentation
of the paper.
\end{acknowledgements}



\end{document}